\documentclass[a4paper]{article}

\usepackage[british]{babel}
\usepackage[utf8]{inputenc}
\usepackage[T1]{fontenc}

\usepackage{amsmath, amsthm, amsfonts, amssymb, amsthm}
\usepackage{mathrsfs}

\usepackage{float}
\usepackage{longtable}
\usepackage{forest}

\usepackage{enumerate}

\usepackage{microtype} 

\newtheorem{theorem}{Theorem}[]
\newtheorem{corollary}[theorem]{Corollary}
\newtheorem{lemma}[theorem]{Lemma}
\newtheorem{proposition}[theorem]{Proposition}
\newtheorem{example}[theorem]{Example}
\newtheorem{conjecture}[theorem]{Conjecture}
\newtheorem{remark}[theorem]{Remark}
\newtheorem{algorithm}[theorem]{Algorithm}

\usepackage{graphicx}
\usepackage[colorinlistoftodos]{todonotes}
\usepackage[colorlinks=true, allcolors=blue]{hyperref}

\title{Internal numerical semigroups}
\author{Mario Casas$^{1,}$\footnote{Electronic
    adress: mariocp11@correo.ugr.es},\ José A. Madrid$^{1,}$\footnote{Electronic
    adress: madrid@ugr.es}\ and J. C. Rosales$^{1,}$\footnote{Electronic
    adress: jrosales@ugr.es}\\
  \small{$^{1}$Department of Algebra, Faculty of Sciences, University of Granada,} \\
  \small{E-18071 Granada, Spain}} 

\date{}

\begin{document}
\maketitle

\begin{center}
  \Large{Version : $19^{th}$ August 2026 }
\end{center}
\begin{abstract}
In this paper the tree structure of numerical semigroups is studied. An internal numerical semigroup is a semigroup located in an internal node of the tree. Analogous a leaf numerical semigroup is placed in a leaf node. Internal semigroups with fixed multiplicity, Frobenius number or genus are studied by providing algorithms to construct all of them. Several conjectures are established, for example, in each case (fixed multiplicity, fixed Frobenius number and fixed genus respectively), the results suggest that there are always more internal than leaf numerical semigroups. Finally, numerical semigroups with fixed multiplicity and Frobenius number simultaneously are investigated. In this case with two invariants fixed, moreover closed formulas to count the number of internal and leaf semigroups are provided for some values of multiplicity and Frobenius number.    
\end{abstract}

\noindent \textbf{Keywords:} Frobenius number, genus, multiplicity, Frobenius variety, internal numerical semigroup.

\section{Introduction}
\label{sec:intro}
Let $\mathbb{Z}$ be the set of integer numbers and $\mathbb{N}=\{x\in\mathbb{Z} \colon x\geq 0\}.$ A submonoid of $(\mathbb{N},+)$ is a subset of $\mathbb{N}$ closed under the addition and containing $0$. A numerical semigroup is a submonoid $S$ of $(\mathbb{N},+)$ such that $\mathbb{N}\backslash S$ is finite.

If $S$ is a numerical semigroup, then $m(S)=\min(S\backslash\{0\})$, $F(S)=\max(\mathbb{Z}\backslash S)$ and $g(S)=\#(\mathbb{N}\backslash S)$, where $\#(A)$ denotes the cardinality of a set $A$, are three important invariants called multiplicity, Frobenius number and genus of $S$ respectively.

If $A$ is a non empty subset of $\mathbb{N}$, then we denote by $\langle A \rangle$ the submonoid of $(\mathbb{N}, +)$  generated by $A$, is to say,
\[
  \langle A\rangle=\left\{ \lambda_1a_1+\ldots +\lambda_n a_n \colon n\in\mathbb{N}\backslash\{0\}, \{a_1, \ldots, a_n\}\subseteq A \text{ and } \{\lambda_1, \ldots, \lambda_n\}\subseteq\mathbb{N}  \right\}.
\]

In Lemma $2.1$ of \cite{Libro} is shown that $\langle A\rangle$ is a numerical semigroup, if and only if, $\gcd(A)=1.$

If $M$ is a submonoid of $(\mathbb{N},+)$ and $M=\langle A\rangle,$ then it is said that $A$ is a system of generators of $M$. Furthermore, if $M\not = \langle B \rangle$ for all $B\subsetneq A$, then we will say that $A$ is a minimal system of generators of $M$. In the Corollary 2.8 of \cite{Libro}, it is shown that all submonoid of $(\mathbb{N},+)$ has a unique minimal system of generators and this system is finite.

We will denote by msg$(M)$ the minimal system of generators of $M$. The cardinal of msg$(M)$ is called the embedding dimension of $M$, denominated as $e(M)$.

The Frobenius problem for numerical semigroups (see \cite{Alfonsin}) consists in
searching for formulas to compute the Frobenius number and the genus of a numerical semigroup based on its minimal system of generators. The aforementioned problem was solved in \cite{Sylvester} for numerical semigroups with embedding dimension equals to two. At present, this problem is still open for numerical semigroups with embedding dimension greater or equal to three.

If $g\in\mathbb{N}$, we denote as $n_g=\#\{ S \text{ such that } S$ is a numerical semigroup with $g(S)=g\}$. The problem to determine $n_g$ has been thoroughly studied in the scientific literature (for example, research such as \cite{Blanco}, \cite{MJPAA}, \cite{MSF}, \cite{Elizalde}, \cite{Kaplan} and \cite{Zhao}).
Many of these publications are motivated for attempting to respond to the 	conjecture established in \cite{MSF}, which argues that $n_g\leq n_{g+1}$.

If $S$ is a numerical semigroups, let denote as $\mu(S)=\#\{x\in\text{msg}(S)$ such that $x > F(S)\}$. An internal numerical semigroup is a numerical semigroup $S$ with $\mu(S)\not = 0$.

We symbolise as $\mathscr{I}=\{S \text{ such that } S \text{ is an internal numerical semigroup}\}$. If $k\in\mathbb{N}$, then we denote as $\mathscr{I}(\text{gen}=k)=\{S\in\mathscr{I} \text{ such that } g(S)=k\}$.
  
In section \ref{sec:arbol}, we will show that if $g\in\mathbb{N}$, then
\[ n_{g+1}=\sum_{S\in\mathscr{I}(\text{gen}=g)} \mu(S).  \]
Consequently, to quantify $n_{g+1}$ is enough to compute $\mathscr{I}(\text{gen}=g)$.

Besides, in section \ref{sec:arbol}, we also conjecture that  $\#\mathscr{I}(\text{gen}=k)\leq \#\mathscr{I}(\text{gen}=k+1)$.

A numerical semigroup $S$ is a leaf if $\mu(S)=0$. Let denote by
$\mathscr{L}=\{S$ such that  $S$  is a leaf numerical semigroup$\}$ and   $\mathscr{L}($gen$=k)=\{S\in\mathscr{L}$  such that $g(S)=k\}$. Furthermore, in section \ref{sec:arbol}, we conjecture that  $\#\mathscr{L}($gen$=k)\leq \#\mathscr{L}($gen$=k+1)$ and  $\#\{\mathscr{L}(\text{gen}=k)\} \leq \#\{\mathscr{I}(\text{gen}=k)\}$ for $k\in\mathbb{N}$.

In section  \ref{sec:fixedgenus}, we will show that $\mathscr{I}$ is a Frobenius variety. This will allow us, by utilising the results of \cite{Variedades}, to provide an algorithm to compute all the elements of $\mathscr{I}(\text{gen}=k)$ for a given $k\in\mathbb{N}$.

If $F$ is a positive integer, then we denote as $\mathscr{I}($Frob$=F)=\{ S\in \mathscr{I}$  such that $F(S)=F\}$. In section \ref{sec:fixedFrobenius}, we will expose that $\mathscr{I}($Frob$=F)$ is a covariety. The results of \cite{Covariedades} will enable us to furnish an algorithm to calculate all the elements of  $\mathscr{I}(\text{Frob}=F)$.

If $m$ is an integer greater or equal to two, then we denote as $\mathscr{I}(\text{mul}=m)=\{ S\in\mathscr{I} \text{ such that } m(S)=m\}$. In section \ref{sec:fixedmultiplicity} we will show that $\mathscr{I}(\text{mul}=m)$ is a Frobenius pseudo-variety. By applying the results of \cite{Pseudo}, we can provide an algorithm able to compute all the elements of $\mathscr{I}(\text{mul}=m, \text{gen}=g)=\{S\in\mathscr{I} \text{ such that } m(S)=m \text{ and } g(S)=g\}$.

If $m$ is an integer greater or equal to two, and $F$ is an integer such that $m-1\leq F$ and $m\nmid F$, then we label as  $\mathscr{I}(\text{mul}=m, \text{Frob}=F)=\{S\in\mathscr{I} \text{ such that } m(S)=m \text{ and } F(S)=F\}$.  In section \ref{sec:fixedmulFrob} we will illustrate that  $\mathscr{I}(\text{mul}=m, \text{Frob}=F)$ is a ratio-covariety. The results of \cite{Ratio} will permit us to write an algorithm to find all the elements of  $\mathscr{I}(\text{mul}=m, \text{Frob}=F)$.

\section{The tree of numerical semigroups}
\label{sec:arbol}

A graph is a pair $G=(V, E)$ where $V$ is a non empty set and $E$ is a subset of $\{(u,v)\in V\times V \text{ such that } u\not = v\}$. The elements of $V$ and $E$ are called vertices and edges respectively.

A path (of length $n$) connecting vertices $u$ and $v$ of $G$ is a sequence of distinct edges of the form $(v_0, v_1)$, $(v_1, v_2), \ldots, (v_{n-1}, v_n)$ such that $u=v_0$ and $v=v_n$.

A graph $G$ is a tree if there exists a vertex $r$ (known as root of $G$) such that for all other vertex $v$ of $G$, there exists a unique path connecting $v$ and $r$.

It is said that a vertex $u$ is a child of vertex $v$ if $(u,v)\in E$. If a vertex $v$ does not have children then we will say that $v$ is a leaf (or external vertex). Otherwise, it is said that $v$ is an internal vertex.

The next result has an immediate proof.
\begin{lemma}\label{Lem:1}
  Let $S$ be a numerical semigroup and $x\in S$. Then the following statements are verified:
  \begin{enumerate}
  \item If $S\not =\mathbb{N}$, then $S\cup\{F(S)\}$ is also a numerical semigroup.
  \item $S\backslash \{x\}$ is a numerical semigroup if and only if $x\in$msg$(S)$.
  \end{enumerate}
\end{lemma}

Let $\mathscr{S}$ the set formed by all numerical semigroups. Then, we define the graph $G(\mathscr{S})$ in the following way: $\mathscr{S}$ is its vertices set and $(S,T)\in\mathscr{S}\times\mathscr{S}$ is an edge if and only if $T=S\cup\{F(S)\}$.

The next result is the Proposition 7.1 of \cite{Libro}.

\begin{proposition}\label{prop:tree}
  $G(\mathscr{S})$ is a tree with $\mathbb{N}$ as root. Besides, the set formed by the children of a vertex $S$ of the tree $G(\mathscr{S})$ are $\{S\backslash\{x\}$ such that $x\in$ msg$(S)$ and $x>F(S)\}$.
\end{proposition}

A tree can be built recursively by starting at the root and connecting each vertex with its children with an edge. By applying Proposition \ref{prop:tree}, we can construct recursively the tree $G(\mathscr{S})$ (see Figure \ref{Fig:semigroupstree}):

\begin{figure}
\forestset{declare toks={elo}{}} 
\begin{center}
  \begin{forest}
    for tree = {edge={<-}}
  [{$\mathbb{N}=<1>$},
    [{$<2,3>$},  edge label={node[midway, right]{$1$}} 
       [{$<3,4,5>$}, edge label={node[midway, left]{$2$}} 
          [{$<4,5,6,7>$}, edge label={node[midway, above left]{$3$}}
             [{$ $}, edge=dotted]
             [{$ $}, edge=dotted]
             [{$ $}, edge=dotted]
             [{$ $}, edge=dotted]
          ]
          [{$<3,5,7>$}, edge label={node[midway, left]{$4$}}
             [{$ $}, edge=dotted]
             [{$ $}, edge=dotted]
          ]
          [{$<3,4>$}, edge label={node[midway, above]{$5$}} ]
       ]
       [{$<2,5>$}, edge label={node[midway, right]{$3$}}
          [{$ $}, no edge]
          [{$ $}, no edge]
          [{$<2,7>$}, edge label={node[midway, right]{$5$}}
             [{$ $}, no edge]
             [{$ $}, no edge]          
             [{$ $}, edge=dotted]
          ]
       ]
  ]  
  ]
\end{forest}
\end{center}
\caption{Portrayal of the tree $G(\mathscr{S})$.}
\label{Fig:semigroupstree}
\end{figure}
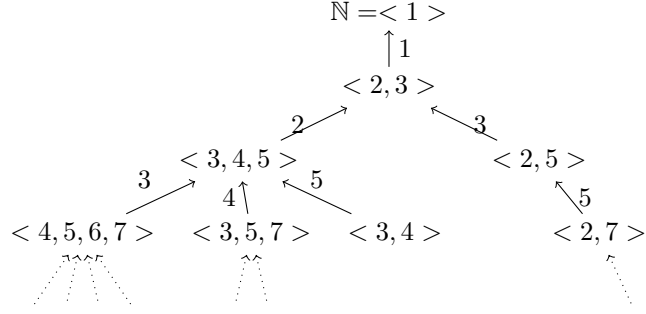

The number $x$ which appears on the edge \raisebox{-4.75ex}{\begin{forest}     for tree = {edge={<-}} [{$P$}, [{$Q$},  edge label={node[midway, right]{$x$}}] ] \end{forest} } indicates that $Q=P\backslash\{x\}.$ Furthermore, note that $F(Q)=x$. Moreover, observe also that $\langle 3, 4\rangle$ is a leaf node and $\langle 1\rangle$,  $\langle 2, 3\rangle$,  $\langle 3, 4, 5\rangle$,  $\langle 2, 5\rangle$,  $\langle 4, 5, 6, 7\rangle$,  $\langle 3, 5, 7\rangle$ and  $\langle 2, 7\rangle$ are internal vertexes of the tree $G(\mathscr{S})$.

If $G=(V, E)$ is a rooted tree and $v\in V$, we can define the depth of $v$, denoted by $d(v)$, as the length of the unique path from $v$ to the root. By definition, we will say that the depth of the root is $0$. If $k\in\mathbb{N}$, then we express as $\mathscr{N}(G(\mathscr{S}), k)=\{v\in V \text{ such that } d(v)=k\}.$

Then $\mathscr{N}$($G(\mathscr{S})$,$0)$=$\{\mathbb{N}\},  \mathscr{N}(G(\mathscr{S}),1)=\{\langle 2, 3\rangle\}, \mathscr{N}(G(\mathscr{S}),2)=\{\langle 3, 4,  5\rangle$,\newline
$ \langle 2, 5\rangle\}$,  $\mathscr{N}(G(\mathscr{S}),3)=\{\langle 4, 5, 6, 7\rangle,\langle 3, 5, 7\rangle, \langle 3, 4\rangle,\langle 2, 7\rangle\}$, $\ldots$

It is straightforward to probe the following result.
\begin{proposition}\label{Prop:3}
  If $k\in\mathbb{N}$, then the following conditions hold:
  \begin{enumerate}
  \item $\mathscr{N}(G(\mathscr{S}), k)=\{ S\in\mathscr{S} \text{ such that } g(S)=k\}.$
   \item $\mathscr{N}(G(\mathscr{S}), k+1)=\{ S\in\mathscr{S}$ such that $S$ is a child of one element of $\mathscr{N}(G(\mathscr{S}), k)$ in the tree  $G(\mathscr{S})  \}.$
   \end{enumerate}
\end{proposition}

As consequence of Proposition \ref{Prop:3} we get the next result

\begin{corollary}\label{Col:4}
\[\#\{ S\in\mathscr{S} \text{ such that } g(S)=k+1\} = \sum_{S\in\mathscr{I}(\text{gen}=k)} \mu(S).  \]
\end{corollary}

\begin{example}
From the Figure \ref{Fig:semigroupstree}, we know that $\mathscr{I}($gen$=3)=\{\langle 4, 5, 6, 7\rangle,\langle 3, 5, 7\rangle,$ \newline $\langle 2, 7\rangle\}$. A simple calculation shows that $\mu(\langle 4, 5, 6, 7\rangle)=4$, $\mu(\langle 3, 5, 7\rangle)=2$ and  $\mu(\langle 2, 7\rangle)=1$. By applying Corollary \ref{Col:4} yields that $\#\{S\in\mathscr{S}$ such that $g(S)=4\}=4+2+1=7$.
\end{example}

From the observation of the Figure \ref{Fig:semigroupstree}, we can establish the following conjectures:

\begin{conjecture}\label{Conj:6}
If $k\in\mathbb{N}$, then $\#\{\mathscr{I}(\text{gen}=k)\} \leq \#\{\mathscr{I}(\text{gen}=k+1)\}$.
\end{conjecture}

\begin{conjecture}\label{Conj:7}
If $k\in\mathbb{N}$, then $\#\{\mathscr{L}(\text{gen}=k)\} \leq \#\{\mathscr{L}(\text{gen}=k+1)\}$.
\end{conjecture}

\begin{conjecture}\label{Conj:8}
If $k\in\mathbb{N}$, then $\#\{\mathscr{L}(\text{gen}=k)\} \leq \#\{\mathscr{I}(\text{gen}=k)\}$.
\end{conjecture}

To finish this section, let us observe that $\{S\in\mathscr{S} \text{ such that } g(S)=k+1\}= \{S\in\mathscr{S} \text{ such that } S \text{ is a child of one element of } \mathscr{I}(\text{gen}=k)\}$.

\section{Internal numerical semigroups with fixed genus}
\label{sec:fixedgenus}

A Frobenius variety is a non-empty family $\mathscr{V}$ of numerical semigroups verifying the next conditions:
\begin{enumerate}
\item If $\{S, T\}\subseteq \mathscr{V}$, then $S\cap T\in\mathscr{V}$.
\item If $S\in\mathscr{V}$ and $S\not = \mathbb{N}$, then $S\cup\{F(S)\}\in\mathscr{V}$.
\end{enumerate}

Next result is well known and  not complicated to proof.
\begin{lemma}\label{Lem:9}
  If $S$ and $T$ are numerical semigroups, then $S\cap T$ is also a numerical semigroup and $F(S\cap T)=\max\{F(S),F(T)\}$.
\end{lemma}

\begin{proposition}\label{Prop:10}
  $\mathscr{I}$ is a Frobenius variety.
\end{proposition}

\begin{proof}
  If $\{S,T\}\subseteq\mathscr{I}$ then, by using Lemma \ref{Lem:9}, we get that $S\cap T$ is a numerical semigroup and $F(S\cap T)=\max \{F(S),F(T)\}$. We can assume without loss of generality that $F(S\cap T)=F(T)$. It is clear that if $x\in\text{msg}(T)$ and $x>F(T)$ then $x\in\text{msg}(S\cap T)$ and $x>F(S\cap T)$. As $\mu(T)\not = 0$ then $\mu(S\cap T)\not = 0$ and hence $S\cap T\in\mathscr{I}$.
  
  If $S\in\mathscr{I}$ and $S\not = \mathbb{N}$, then by using the Lemma \ref{Lem:1}, we know that $S\cup\{F(S)\}$ is a numerical semigroup. It is clear that $F(S)\in\text{msg}(S\cup\{F(S)\})$ and $F(S)>F(S\cup\{F(S)\})$. Therefore $\mu(S\cup\{F(S)\})\not = 0$ and so $S\cup\{F(S)\}\in\mathscr{I}$.
\end{proof}

Let define the graph $G(\mathscr{I})$ as follows: $\mathscr{I}$ is the set of vertexes and $(S,T)\in\mathscr{I}\times\mathscr{I}$ is a edge if and only if $T=S\cup\{F(S)\}$.

By applying [the Proposition 24 and the Theorem 27 of \cite{Variedades}] we obtain the next result.

\begin{theorem}
  The graph $G(\mathscr{I})$ is a tree and $\mathbb{N}$ is its root. Besides, the set formed with the children of a vertex $S$ of the tree $G(\mathscr{I})$ is $\{S\backslash\{x\} \text{ such that } x\in\text{msg}(S),\; x>F(S) \text{ and } S\backslash\{x\}\in\mathscr{I}\}$. 
\end{theorem}

The next lemma can be deduced straightforwardly from Corollary 18 of \cite{frases}.

\begin{lemma}
  Let $S$ be a numerical semigroup and $x\in\text{msg}(S)$ such that $x>F(S)$. Then the following conditions are satisfied:
  \begin{enumerate}
  \item If $x=\text{m}(S)$, then $\text{msg}(S\backslash\{x\})=\{x+1, x+2,\ldots, 2x+1\}$.
  \item If $x\not =\text{m}(S)$, then
  \begin{equation*} \text{msg}(S\backslash\{x\}) = \begin{cases}\text{msg}(S)\backslash\{x\} & \text{if } x+m(S)-y\in S \\
      & \text{for some }\\
      &y\in \text{msg}(S)\backslash\{x,m(S)\} ,\\
        (\text{msg}(S)\backslash\{x\})\cup\{x+m(S)\}  & \text{in another case.} \end{cases}
      \end{equation*}
  \end{enumerate}
\end{lemma}

\begin{lemma} It is hold that
  \[  \{g(S) \text{ such that } S\in\mathscr{I}\}=\mathbb{N}.  \]
\end{lemma}

\begin{proof}
  It is enough to observe that $S=\{0, g+1,\to\}\in\mathscr{I}$ and $g(S)=g$.
\end{proof}

At this point, we can provide an algorithm to compute all the internal numerical semigroups with a given genus $g$.

\begin{algorithm}  \label{alg:fixedgenus}
  Computation of $\mathscr{I}(gen=g)$, i.e., all internal numerical semigroups with fixed genus g. \\ 
\; \\
\textbf{Input:} A non negative integer $g$. \\
\textbf{Output:} $\mathscr{I}(\text{gen}=g)$. \\
\begin{enumerate}[(1)]
\item  $A=\{\mathbb{N}\}, \; i=0$. 
\item  \textbf{if} {$i=g$} \textbf{then return $A$}. 
\item  \textbf{end if} 
\item  \textbf{for all} {$S\in A$} \textbf{do} 
\item  \hspace{0.4cm} Compute $\alpha(S)=\{x\in$msg$(S)$ such that $x>F(S)$ and $\mu(S\backslash\{x\})\not = 0\}$. 
\item  \textbf{end for} 
\item  $A= \bigcup\limits_{S\in A} \{S\backslash\{x\}$ such that $x\in\alpha(S)\}$. 
\item  $i=i+1$. 
\item  \textbf{go to} 2.
\end{enumerate}  
\end{algorithm}

Let us see an example to show how this algorithm works.

\begin{example}
  We want to compute $\mathscr{I}($gen$=3)$ by applying the Algorithm \ref{alg:fixedgenus}:
  \begin{itemize}
  \item $A=\{\mathbb{N}\}, \; i=0$.
  \item $\alpha(\mathbb{N})=\{1\}$.
  \item $A=\{\langle 2, 3\rangle\},\; i=1$.
  \item  $\alpha(\langle 2, 3\rangle)=\{2, 3\}$.
  \item $A=\{\langle 3, 4, 5\rangle, \langle 2, 5\rangle\},\; i=2$.
  \item  $\alpha(\langle 3, 4, 5\rangle)=\{3, 4\}$ and $\alpha(\langle 2, 5\rangle)=\{5\}$.
  \item  $A=\{\langle 4, 5, 6, 7\rangle, \langle 3, 5, 7\rangle, \langle 2, 7\rangle\},\; i=3$.
  \end{itemize}
  The Algorithm \ref{alg:fixedgenus} returns $\mathscr{I}($gen$=3)=\{\langle 4, 5, 6, 7\rangle, \langle 3, 5, 7\rangle, \langle 2, 7\rangle\}$.
\end{example}

Performing computations by applying the Algorithm \ref{alg:fixedgenus}, we obtain the Table \ref{tab:fixedgenus}, where $ins_g$ denotes the number of internal numerical semigroups with genus $g$, $lns_g$ the number of leaf numerical semigroups with genus $g$ and $ns_g$ is the number of numerical semigroups with genus $g$, is to say, $ns_g=ins_g+lns_g$. 

\begin{table}[htb]
\makebox[1 \textwidth][c]{       
\resizebox{1.3 \textwidth}{!}{   
  
\begin{tabular}{|c c c c | c c c | c c c|}
\hline  
$g$ & $ins_g$ & $lns_g$ & $ns_g$ &  $ins_{g-1}+ins_{g-2}$ & $\frac{ins_{g-1}+ins_{g-2}}{ins_g}$ & $\frac{ins_{g}}{ins_{g-1}}$ &    $lns_{g-1}+lns_{g-2}$ & $\frac{lns_{g-1}+lns_{g-2}}{lns_g}$ & $\frac{lns_{g}}{lns_{g-1}}$ \\ \hline
0 &1     &0     &1     &      &        &        &      &        &   \\
1 &1     &0     &1     &      &        &  1     &      &        &  \\
2 &2     &0     &2     & 2    & 1      &   2    &0     &        &  \\
3 &3     &1     &4     & 3    & 1      & 1.5    &0     & 0      &  \\
4 &5     &2     &7     & 5    &  1     &1.666667&1     &  0.5   &  2 \\
5 &10    &2     &12    & 8    &  0.8   & 2      &3     &1.5     &  1 \\
6 &15    &8     &23    & 15   &  1     & 1.5    & 4    & 0.5    &  4 \\
7 &27    &12    &39    & 25   &0.925926&1.8     &10    &0.833333&  1.5  \\
8 &47    &20    &67    & 42   &0.893617&1.740741&20    & 1      &1.666667 \\
9 &81    &37    &118   & 74   &0.91358 &1.723404& 32   &0.864865&  1.85\\
10&132   &72    &204   & 128  &0.969697&1.62963 &57    &0.791667&1.945946 \\
11&233   &110   &343   & 213  &0.914163&1.765152&109   &0.990909&1.527778  \\
12&381   &211   &592   & 365  &0.958005&1.635193&182   &0.862559&1.918182  \\
13&651   &350   &1001  & 614  &0.943164&1.708661&321   &0.917143&1.658768   \\       
14&1103  &590   &1693  & 1032 &0.93563 &1.694316&561   &0.950847&1.685714   \\
15&1836  &1021  &2857  & 1754 &0.955338&1.664551&940   &0.920666&1.730508 \\
16&3064  &1742  &4806  & 2939 &0.959204&1.668845&1611  &0.924799&1.70617 \\
17&5144  &2901  &8045  &  4900&0.952566&1.678851&2763  &0.95243 &1.665327  \\
18&8540  &4927  &13467 & 8208 &0.961124&1.660187& 4643 &0.942358&1.69838  \\
19&14220 &8244  &22464 & 13684&0.962307&1.665105&7828  &0.949539&1.673229  \\
20&23646 &13750 &37396 &22760 &0.962531&1.662869&13171 &0.957891&1.66788 \\
21&39235 &22959 &62194 &37866 &0.965108&1.659266&21994 &0.957969&1.669745 \\
22&64890 &38356 &103246&62881 &0.96904 &1.65388 &36709 &0.95706 &1.67063\\
23&107403&63560 &170963&104125&0.969479&1.655155&61315 &0.964679&1.657107 \\
24&177349&105479&282828&172293&0.971491&1.651248&101916&0.966221&1.659519  \\
25&292551&174673&467224&284752&0.973341&1.649578&169039&0.967745&1.655998 \\
26&482133&288699&770832&469900&0.974627&1.648031&280152&0.970395&1.652797  \\
  \hline
\end{tabular}
} 
} 

\caption{\label{tab:fixedgenus}  Computational results on the number of internal numerical semigroups up to genus 26, where $ins_g$ denotes the number of internal numerical semigroups with genus $g$, $lns_g$ the number of leaf numerical semigroups with genus $g$ and $ns_g$ is the total number of numerical semigroups with genus $g$.}
\end{table}

The results depicted in Table \ref{tab:fixedgenus} supports conjectures \ref{Conj:6}, \ref{Conj:7} and \ref{Conj:8}. Moreover, it seems that the number of internal and leaf numerical semigroups have a Fibonacci-like property analogous to the property established in \cite{Bras-Amorós2008}.
Following these ideas, we have written the number of internal numerical semigroups ($ins_g$) of the given genus $g$, the Fibonacci-like-estimated value given by the sum of the number of internal numerical semigroups of the two previous genera, the value of the quotient $\frac{ins_{g-1}+ins_{g-2}}{ins_g}$, and the value of the quotient $\frac{ins_g}{ins_{g-1}}$ up to genus $26$. Besides, we have done the same for leaf numerical semigroups ($lns_{g})$.
The values of these quotients follow an analogous behaviour than the computations performed in \cite{Bras-Amorós2008}, so we could conjecture that  $\frac{ins_{g-1}+ins_{g-2}}{ins_g}$ and  $\frac{lns_{g-1}+lns_{g-2}}{lns_g}$ tend to $1$ when $g$ tends to infinity. On the other hand,  $\frac{ins_g}{ins_{g-1}}$ and  $\frac{lns_g}{lns_{g-1}}$ approach the golden ratio as g approaches infinity. So, we can write down analogous  conjectures (to those provided in \cite{Bras-Amorós2008}) for internal and leaf numerical semigroups as follows:
\begin{conjecture}\label{Conj:Fibonacci}
   The number of internal and leaf numerical semigroups have a Fibonacci-like property, it is to say:
\[ins_{g}\geq ins_{g-1} + ins_{g-2} \text{, for } g\geq 2,\]
and
\[lns_{g}\geq lns_{g-1} + lns_{g-2} \text{, for } g\geq 6. \]
Also, the asymptotic behaviour is
\begin{itemize}
\item $\lim_{g\to \infty}\frac{ins_{g-1}+ins_{g-2}}{ins_g}=1,$
\item $\lim_{g\to \infty}\frac{lns_{g-1}+lns_{g-2}}{lns_g}=1,$
\item $\lim_{g\to \infty}  \frac{ins_g}{ins_{g-1}}=\varphi$,
 \item $\lim_{g\to \infty}  \frac{lns_g}{lns_{g-1}}=\varphi$. 
 \end{itemize}
  where $\varphi$ is the golden ratio.
\end{conjecture}

\begin{remark}
 Conjecture \ref{Conj:Fibonacci} is stronger than conjectures \ref{Conj:6} and \ref{Conj:7}.
\end{remark}

\section{Internal numerical semigroups with fixed Frobenius number}
\label{sec:fixedFrobenius}

In this section $F$ will denote a positive integer, $\mathscr{I}(\text{Frob}=F)=\{S\in\mathscr{I}$ such that $F(S)=F\}$ and $\mathscr{L}(\text{Frob}=F)=\{S\in\mathscr{L}$ such that $F(S)=F\}$.

A covariety is a family $\mathscr{C}$ of numerical semigroups fulfilling the following conditions:
\begin{enumerate}
\item $\mathscr{C}$ has a minimum (respects to the order provided by the inclusion).
\item If $\{S, T\}\subseteq \mathscr{C}$ then $S\cap T\in\mathscr{C}$.
\item If $S\in\mathscr{C}$ and $S\not = \min(\mathscr{C})$, then $S\backslash \{m(S)\}\in\mathscr{C}$.
\end{enumerate}

\begin{proposition}\label{Prop:16}
$\mathscr{I}(\text{Frob}=F)$ is a covariety.
\end{proposition}

\begin{proof}
It is clear that $\Delta(F)=\{0, F+1, \to\}$ is the minimum of $\mathscr{I}(\text{Frob}=F)$.

By applying the Lemma \ref{Lem:9} and Proposition \ref{Prop:10} we can straightforward deduce that if $\{S, T\}\subseteq \mathscr{I}(\text{Frob}=F)$, then $S\cap T\in \mathscr{I}(\text{Frob}=F)$.

Let $S$ be a numerical semigroup verifying that $S\in \mathscr{I}(\text{Frob}=F)$ and $S\not = \Delta(F)$. Then $m(S)<F$ and using the Lemma \ref{Lem:1} we get  that $S\backslash \{m(S)\}$ is a numerical semigroup with Frobenius number $F$. If $x\in\text{msg}(S)$ and $x>F$, then it is clear that $x\in\text{msg}(S\backslash\{m(S)\})$ and $x>F(S\backslash\{m(S)\})$. Therefore $\mu(S\backslash\{m(S)\})\not = 0$ and as a result $S\backslash\{m(S)\}\in \mathscr{I}(\text{Frob}=F)$.
\end{proof}

Let define the graph $G(\mathscr{I}(\text{Frob}=F))$ in the following way: $\mathscr{I}(\text{Frob}=F)$ is its set of vertexes and $(S,T)\in \mathscr{I}(\text{Frob}=F)\times \mathscr{I}(\text{Frob}=F)$ is a edge if an only if $T=S\backslash\{m(S)\}$.

By applying the previous Proposition \ref{Prop:16} and [the Proposition 2.3 from  \cite{Covariedades}] the following result can be obtained.

\begin{proposition}
 $G(\mathscr{I}(\text{Frob}=F))$ is a tree and $\Delta(F)$ is its root.
\end{proposition}

Now, our goal is to study the children of an arbitrary vertex of the tree $G(\mathscr{I}(\text{Frob}=F))$, so we introduce the next notion.

An integer number $x$ is called a special gap of a numerical semigroup $S$ if $x\not\in S$ and $S\cup\{x\}$ is a numerical semigroup. Let denote by $SG(S)$ to the set containing all the special gaps of $S$.

By using the previous Proposition \ref{Prop:16} and the [Proposition 2.4 from  \cite{Covariedades}] one can get the next result.

\begin{proposition}
  If $S\in \mathscr{I}(\text{Frob}=F)$, then the set established by all the children of $S$ in the tree  $G(\mathscr{I}(\text{Frob}=F))$ is
  \[ \{S\cup\{x\} \text{ such that } x\in SG(S),\; x<m(S), \; x\not=F \text{ and } \mu(S\cup\{x\})\not=0\}. \]
\end{proposition}

If $S$ is a numerical semigroup and $n\in S\backslash\{0\}$, then we define (in honour of \cite{Apery}) the Apéry set of $S$ with respect to $n$ as $Ap(S,n)=\{s\in S \text{ such that } s-n\not\in S\}$.

The next result is the Lemma 2.4 of \cite{Libro}.

\begin{lemma}
  If $S$ is a numerical semigroup and $n\in S\backslash\{0\}$, then $Ap(S,n)=\{w(0), w(1), \ldots, w(n-1)\}$ where $w(i)$ is the lesser element of $S$ which is congruent with $i$ module $n$.
\end{lemma}
  
Before of providing an algorithm to assemble $\mathscr{I}(\text{Frob}=F)$, we will provide a remark which will help us to perform the computations of the algorithm.
  
\begin{remark}
  If $S$ is a numerical semigroup and we are aware of $Ap(S,n)=\{w(0), w(1), \ldots, w(n-1)\}$ for some $n\in S\backslash\{0\}$, then:
  \begin{enumerate}
  \item The problem of belonging to $S$ is solved as an integer $x$ belongs to $S$ if and only if $x\geq w(x\mod n)$.
  \item By using the Remark 1 of \cite{Covariedades} one can compute $SG(S)$.
  \item By utilising the Remark 2 of \cite{Covariedades} we can calculate $Ap(S\cup\{x\},n)$ for all $x\in SG(S)$.
  \item We can construct msg$(S)$ as it is straightforward to observe that
  \begin{flalign*}  
    \text{msg}(S)=&\{ x\in(Ap(S,n)\backslash\{0\})\cup\{n\} \text{ such that } \\
    &x\not \in \left( (Ap(S,n)\backslash\{0\}) \cup\{n\}\right) + \left( (Ap(S,n)\backslash\{0\}) \cup\{n\}\right) \}. &&
\end{flalign*}

  \end{enumerate}
\end{remark}

At this point, we are ready to establish the following algorithm.

\begin{algorithm}  \label{alg:fixedFrobenius}
Computation of $\mathscr{I}(\text{Frob}=F)$, i.e., all internal numerical semigroups with fixed Frobenius number F. \\
\; \\
\textbf{Input:} A positive integer $F$. \\ 
\textbf{Output:} $\mathscr{I}(\text{Frob}=F)$. 
\begin{enumerate}[(1)]
\item  $A=\{\Delta(F)\}, \; B=\{\Delta(F)\}$ and $Ap(\Delta(F),F+1)=\{0, F+2,\ldots, 2F+1\}$. 
\item  \textbf{for all} {$S\in B$} \textbf{do} 
\item  \hspace{0.5cm} Compute $\theta(S)=\{x\in SG(S)$ such that $x<m(S)$, $x\not= F$ and $\mu(S\cup\{x\})\not = 0\}$. 
\item  \textbf{end for} 
\item  \textbf{i}f {$\bigcup\limits_{S\in B} \theta(S) =\emptyset$} \textbf{then return} $A$. 
\item  \textbf{end if} 
\item  $C= \bigcup\limits_{S\in B} \{S\cup\{x\}$ such that $x\in\theta(S)\}$. 
\item  $A=A\cup C$ and $B=C$. 
\item  \textbf{for all} {$S\in B$} \textbf{do} 
\item  \hspace{0.5cm} Compute $Ap(S,F+1)$. 
\item  \textbf{end for} 
\item  \textbf{go to} 2.
  \end{enumerate}
\end{algorithm}

We are using an example to show how to apply the Algorithm \ref{alg:fixedFrobenius}.

\begin{example}
  We are going to compute $\mathscr{I}($Frob$=5)$ by utilising the Algorithm \ref{alg:fixedFrobenius}.
  \begin{itemize}
  \item $A=\{\Delta(5)\}, \; B=\{\Delta(5)\}$ and $Ap(\Delta(5),6)=\{0, 7, 8, 9, 10, 11\}$.
  \item $\theta(\Delta(5))=\{3, 4\}$.
  \item $C=\{\Delta(5)\cup\{3\}, \Delta(5)\cup\{4\}\}$.
  \item $A=\{\Delta(5), \Delta(5)\cup\{3\}, \Delta(5)\cup\{4\}\}$ and $B=\{\Delta(5)\cup\{3\}, \Delta(5)\cup\{4\}\}$.
  \item $Ap(\Delta(5)\cup\{3\},6)$=$\{0, 3, 7, 8, 10, 11\}$ and $Ap(\Delta(5)\cup\{4\},6)$=$\{0, 4, 7, 8, 9, 11\}$.
  \item $\theta(\Delta(5)\cup\{3\})=\emptyset$ and $\theta(\Delta(5)\cup\{4\})=\{2\}$.
  \item $C=\{\Delta(5)\cup\{2, 4\}\}$.
  \item $A=\{\Delta(5), \Delta(5)\cup\{3\}, \Delta(5)\cup\{4\}, \Delta(5)\cup\{2,4\}\}$ and $B=\{\Delta(5)\cup\{2, 4\}\}$.
  \item $Ap(\Delta(5)\cup\{2, 4\},6)=\{0, 2, 4, 7, 9, 11\}$.
  \item $\theta(\Delta(5)\cup\{2, 4\})=\emptyset$.
  \end{itemize}
  The  Algorithm \ref{alg:fixedFrobenius} returns   $\mathscr{I}($Frob$=5)=\{\Delta(5), \Delta(5)\cup\{3\}, \Delta(5)\cup\{4\}, \Delta(5)\cup\{2,4\}\}$. 
\end{example}

Performing computations by applying the Algorithm \ref{alg:fixedFrobenius}, we achieve the Table \ref{tab:fixedFrobenius}, where $ins_F$ denotes the number of internal numerical semigroups with Frobenius number $F$, $lns_F$ the number of leaf numerical semigroups with Frobenius number $F$ and $ns_F$ is the number of numerical semigroups with Frobenius number $F$, is to say, $ns_F=ins_F+lns_F$. 

\begin{table}[htb]
\centering
\begin{tabular}{|c c c c | c c c c|}
\hline  
$F$ & $ins_F$ & $lns_F$ & $ns_F$ &  $F$ & $ins_F$ & $lns_F$ & $ns_F$ \\ \hline
  1 & 1 & 0 & 1 & 29 & 29165 & 5738 & 34903\\
  2 & 1 & 0 & 1 & 30 & 29816 & 2006 & 31822 \\
  3 & 2 & 0 & 2 & 31 & 59348 & 11506 & 70854 \\
  4 & 2 & 0 & 2 & 32 & 63912 & 4769 & 68681 \\
  5 & 4 & 1 & 5 & 33 & 116284 & 21107 & 137391 \\
  6 & 4 & 0 & 4 & 34 & 131020 & 9641 & 140661 \\
  7 & 9 & 2 & 11 & 35 & 246157 & 45924 & 292081 \\
  8 & 10 & 0 & 10 & 36 & 253249 & 17009 & 270258 \\
  9 & 17& 4 & 21 & 37 & 499275 & 92168 & 591443 \\
  10 & 21 & 1 & 22 & 38 & 542253 & 40200 & 582453 \\
  11 & 42 & 9 & 51 & 39 & 982946 & 173066 & 1156012 \\
  12 & 38 & 2 & 40 & 40 & 1084756 & 76563 & 1161319 \\
  13 & 85 & 21 & 106 & 41 & 2057338 & 368373 & 2425711  \\
  14 & 98 & 5 & 103 & 42 & 2144620 & 142583 & 2287203  \\
  15 & 167 & 33 & 200 & 43 & 4155812 & 733622 & 4889434  \\
  16 & 193 & 12 & 205 & 44 & 4468559 & 317112 & 4785671  \\
  17 & 379 & 86 & 465 & 45 & 8172756 & 1402411 & 9575167  \\
  18 & 383 & 22 & 405 & 46 & 9045603 & 633241 & 9678844  \\
  19 & 787 & 174 & 961 & 47 & 16984804 & 2935098 & 19919902  \\
  20 & 838 & 62 & 900 & 48 & 17728813 & 1168079 & 18896892  \\
  21 & 1529 & 299 & 1828 & 49 & 34170025 & 5840826 & 40010851  \\
  22 & 1781 & 132 & 1913 & 50 & 36870123 & 2575763 & 39445886  \\
  23 & 3390 & 706 & 4096 & 51 & 67524277 & 11270000 & 78794277  \\
  24 & 3378  & 200  & 3578 & 52 & 73913113 & 5017193 & 78930306  \\
  25 & 6855 & 1418 & 8273 & 53 & 138945789 & 23360286 & 162306075  \\
  26 & 7588 & 587 & 8175  & 54 & 146428617 & 9579565 & 156008182  \\
  27 & 13554 & 2578 & 16132  & 55 & 279182572 & 46617725 & 325800297  \\
  28 & 15165 & 1102 & 16267 & 56 & 300272713 & 20250566 & 320523279  \\
\hline
\end{tabular}
\caption{\label{tab:fixedFrobenius}  Computational results on the number of internal numerical semigroups up to Frobenius number 56, where $ins_F$ denotes the number of internal numerical semigroups with Frobenius number $F$, $lns_F$ the number of leaf numerical semigroups with Frobenius number $F$ and $ns_F$ is the total number of numerical semigroups with Frobenius number $F$.}
\end{table}

The results depicted in Table \ref{tab:fixedFrobenius} suggest us two new conjectures that we can establish as: 

\begin{conjecture}
\[ins_{F}> lns_{F} \text{, for } F\geq 1.\]
\end{conjecture}

\begin{conjecture}
\[lns_{F}> lns_{F+1}  \text{, for } F \text{ odd, and } F\geq 5. \]
\end{conjecture}

\section{Internal numerical semigroups with fixed multiplicity}
\label{sec:fixedmultiplicity}

Along this section $m$ will denote an integer greater or equal to two, $\mathscr{I}(\text{mul}=m)=\{S\in\mathscr{I}$ such that $m(S)=m\}$ and $\mathscr{L}(\text{mul}=m)=\{S\in\mathscr{L}$ such that $m(S)=m\}$.

A Frobenius pseudo-variety is a family $\mathscr{P}$ of numerical semigroups fulfilling the following conditions:
\begin{enumerate}
\item $\mathscr{P}$ has a maximum (respects to the order induced by the inclusion).
\item If $\{S, T\}\subseteq\mathscr{P}$, then $S\cap T\in\mathscr{P}$.
\item If $S\in\mathscr{P}$ and $S\not = \max(\mathscr{P})$, then $S\cup\{F(S)\}\in\mathscr{P}$.
\end{enumerate}

\begin{proposition}
 $\mathscr{I}(\text{mul}=m)$ is a Frobenius pseudo-variety.
\end{proposition}

\begin{proof}
\;
  
\begin{enumerate}
\item   It is straightforward to check that $\Delta(m-1)=\{0, m,\to\}$ is the maximum of $\mathscr{I}(\text{mul}=m)$.
\item If $\{S,T\}\subseteq \mathscr{I}(\text{mul}=m)$, then $\{S,T\}\subseteq\mathscr{I}$ and $m(S)=m(T)=m$. Therefore $m(S\cap T)=m$ and by Proposition \ref{Prop:10} we get that  $S\cap T\in\mathscr{I}$. So   $S\cap T \in \mathscr{I}(\text{mul}=m)$.
  \item If $S\in\mathscr{I}(\text{mul}=m)$ and $S\not =\Delta(m-1)$, then $m < F(S)$. By using Proposition \ref{Prop:10} we get that $S\cup\{F(S)\}\in\mathscr{J}$. As $m(S\cup\{F(S)\})=m$ then  $S\cup\{F(S)\}\in\mathscr{I}(\text{mul}=m)$.
\end{enumerate}
\end{proof}

We define the graph $G(\mathscr{I}(\text{mul}=m))$ as follows: its set of vertexes is $\mathscr{I}(\text{mul}=m)$ and $(S,T)\in\mathscr{I}(\text{mul}=m)\times \mathscr{I}(\text{mul}=m)$ is a edge if and only if $T=S\cup\{F(S)\}$. 

The next result can be deduced from Theorem 3 of \cite{Pseudo}.

\begin{theorem}
  The graph $G(\mathscr{I}(\text{mul}=m))$ is a tree with root $\Delta(m-1)$. Besides, the set formed with all the children of a given vertex $S$ of the tree  $G(\mathscr{I}(\text{mul}=m))$ is $\{S\backslash\{x\} \text{ such that } x\in\text{msg}(S), \; x\not = m, \; x>F(S), \text{ and } \mu(S\backslash\{x\})\not = 0\}$.
\end{theorem}

\begin{proposition}
  \[ \{g(S) \text{ such that } S\in \mathscr{I}(\text{mul}=m)\} = \{m-1, \to\}. \]
\end{proposition}

\begin{proof}
It is clear that if $m(S)=m$, then $g(S)\geq m-1$. To conclude this proof we have to check that whether there exists $S\in\mathscr{I}(\text{mul}=m)$ such that $g(S)=a$ with $a\in\{m-1, \to\}$. Indeed, $a$ can be expressed as $a=q\cdot(m-1)+r$ with $q\in\mathbb{N}\backslash\{0\}$ and $r\in\{0,\ldots,m-2\}$. It is straightforward that $S=\{0, m, 2m, \ldots, qm, qm+r+1, \to\}\in \mathscr{I}(\text{mul}=m)$ and $g(S)=q(m-1)+r=a$.
\end{proof}

Now, we have all the tools needed to provide an algorithm to obtain all the internal numerical semigroups with multiplicity $m$ and genus $g$. Next, we present this algorithm.

\begin{algorithm} \label{alg:fixedmultiplicity}
  Computation of $\mathscr{I}(\text{mul}=m,\text{ gen}=g)$, i.e., all internal numerical semigroups with multiplicity $m$ and genus $g$. \\
 \; \\
\textbf{Input:} An integer $g$ fulfilling that $g\geq m-1$.  \\
\textbf{Output:}  $\mathscr{I}(\text{mul}=m,\text{ gen}=g)=\{S\in\mathscr{I}(\text{mul}=m)\text{ such that } g(S)=g\}$. 
\begin{enumerate}[(1)]
\item  $A=\{\Delta(m-1)\}, \; i=m-1$.
\item  \textbf{if} {$i=g$} \textbf{then return} $A$. 
\item  \textbf{end if} 
\item  \textbf{for all} {$S\in A$} \textbf{do} 
\item  \hspace{0.4cm} Compute $\gamma(S)=\{x\in \text{msg}(S)\backslash\{m\}$ such that $ x>F(S)$  and  $\mu(S\backslash\{x\})\not = 0\}$. 
\item  \textbf{end for} 
\item  $A= \bigcup\limits_{S\in A} \{S\backslash\{x\} \text{ such that }x\in\gamma(S)\}$. 
\item  $i=i+1$. 
\item  \textbf{go to} 2.
  \end{enumerate}
\end{algorithm}

We will provide an example to show how the Algorithm \ref{alg:fixedmultiplicity} works.

\begin{example}
  We are going to compute  $\mathscr{I}(\text{mul}=5,\text{ gen}=7)$ by applying the Algorithm \ref{alg:fixedmultiplicity}.
  \begin{itemize}
  \item $A=\{\Delta(4)\}, \; i=4$.
  \item $\gamma(\Delta(4))=\{6, 7, 8\}$.
  \item $A=\{\langle 5, 7, 8, 9, 11\rangle, \; \langle 5, 6, 8, 9\rangle, \; \langle 5, 6, 7, 9\rangle\}, \; i=5$.
  \item  $\gamma=(\langle 5, 7, 8, 9, 11\rangle)=\{ 7, 8, 9\}, \; \gamma(\langle 5, 6, 8, 9\rangle)=\{8\}, \; \gamma(\langle 5, 6, 7, 9\rangle)=\emptyset$.
  \item $A=\{\langle 5, 8, 9, 11, 12\rangle, \; \langle 5, 7, 9, 11, 13\rangle, \; \langle 5, 7, 8, 11\rangle, \; \langle 5, 6, 9, 13\rangle\}, \; i=6$.
  \item $\gamma(\langle 5, 8, 9, 11, 12\rangle)$=$\{ 8, 9, 11\}$,$ \; \gamma(\langle 5, 7, 9, 11, 13\rangle)$=$\{9, 11\}$,$ \; \gamma(\langle 5, 7, 8, 11\rangle)$=$\emptyset$, and $\gamma(\langle 5, 6, 9, 13\rangle)$=$\{9\}$.
  \item $A=\{\langle 5, 9, 11, 12, 13\rangle, \langle 5, 8, 11, 12, 14\rangle, \langle 5, 8, 9, 12\rangle$, $\langle 5, 7, 11, 13\rangle$, $\langle 5, 7, 9, 13\rangle$, $\langle 5, 6, 13, 14\rangle \}$, $i=7$.
  \end{itemize}
  The  Algorithm \ref{alg:fixedmultiplicity} returns
  \begin{equation*}
    \begin{split}
      \mathscr{I}(\text{mul}=5,\text{ gen}=7)= &\{\langle 5, 9, 11, 12, 13\rangle, \; \langle 5, 8, 11, 12, 14\rangle, \; \langle 5, 8, 9, 12\rangle, \; \\
      & \langle 5, 7, 11, 13\rangle,  \; \langle 5, 7, 9, 13\rangle,   \; \langle 5, 6, 13, 14\rangle \}.
      \end{split}
\end{equation*}    
\end{example}

   Performing computations by applying the Algorithm  \ref{alg:fixedmultiplicity}, we attain the Table \ref{tab:fixedmultiplicity-m5m8} for fixed multiplicities equals to $5$ (left column) and $8$ (right column), where $ins_{gm}$ denotes the number of internal numerical semigroups with
genus $g$ and multiplicity $m$, $lns_{gm}$ the number of leaf numerical semigroups with genus $g$ and multiplicity $m$, and $ns_{gm}$ is the number of numerical semigroups with genus $g$ and multiplicity $m$, is to say, $ns_{gm} = ins_{gm} + lns_{gm}$. 

\begin{table}[H] 
\begin{center}
  \begin{tabular}{|c || c c c | c c c|}
\hline
     &   \multicolumn{3}{ |c| }{$m=5$} & \multicolumn{3}{ |c| }{$m=8$} \\
  $g$& $ins_{gm}$ & $lns_{gm}$ & $ns_{gm}$  & $ins_{gm}$ & $lns_{gm}$ & $ns_{gm}$ \\ \hline
 4 & 1  & 0  & 1  &    &    &  \\ 
 5 & 3  & 1  & 4  &    &    &   \\ 
 6 & 4  & 3  & 7  &    &    &   \\ 
 7 & 6  & 4  & 10 & 1  & 0  & 1 \\ 
 8 & 7  & 6  & 13 & 6  & 1  & 7   \\ 
 9 & 12 & 4  & 16 & 16 & 6  & 22  \\ 
10 & 11 & 11 & 22 & 28 & 16 & 44 \\ 
11 & 18 & 6  & 24 & 47 & 25 & 72 \\ 
12 & 18 & 14 & 32 & 71 & 45 &116  \\  
13 & 24 & 11 & 35 &102 & 70 &172  \\   
14 &30  &13  & 43 &141 &98  &239  \\  
15 & 37 &14  &51  &198 &127 &325  \\  
16 &40  &21  &61  &256 &185 &441   \\ 
17 &50  &18  &68  &337 &236 &573 \\   
18 &55  &25  &80  &436 &301 &737  \\ 
19 &67  &22  &89  &564 &381 &945  \\ 
20 &75  &29  &104 &711 &482 &1193  \\ 
21 &90  &25  &115 &889 &601 &1490  \\  
22 &94  &39  &133 &1094&753 &1847  \\  
23 &113 &30  &143 &1365&897 &2262  \\  
24 &124 &40  &164 &1681&1085&2766 \\   
25 &140 &41  &181 &2046&1307&3353 \\  
26 &156 &45  &201 &2473&1562&4035  \\ 
\hline
  \end{tabular}
      \caption{ Computational results on the number of internal numerical semigroups up to genus 26 for multiplicities $5$ (left column) and $8$ (right column), where $g$ denotes the genus of the numerical semigroup, $m$ its multiplicity,  $ins_{gm}$ is the number of internal numerical semigroups with genus $g$ and multiplicity $m$,  $lns_{gm}$ the number of leaf numerical semigroups with genus $g$ and multiplicity $m$,  $ns_{gm}$ the total number of numerical semigroups with genus $g$ and multiplicity $m$.}
\label{tab:fixedmultiplicity-m5m8}   
\end{center}
\end{table}

From the data depicted in the Table \ref{tab:fixedmultiplicity-m5m8} we can establish the next conjecture:

\begin{conjecture}
  \[ ns_{gm} \leq ns_{(g+1)m},  \; \mbox{ for } m\geq 2 \mbox{ with } g\geq m-1. \]
\end{conjecture}  

On the other hand, we observe that  $\#\mathscr{I}(\text{mul}=5,\text{ gen}=9)=12 \not\leq 11 = \#\mathscr{I}(\text{mul}=5,\text{ gen}=10)$ and $\#\mathscr{L}(\text{mul}=5,\text{ gen}=22)=39 \not\leq 30 = \#\mathscr{L}(\text{mul}=5,\text{ gen}=23)$. Anyway, reviewing our data, where all the internal and leaf numerical semigroups have been computed up to genus $26$ for all available multiplicities, we can state two new conjectures.

\begin{conjecture}
  If $k\geq m-1$ and $m\geq 7$, then
  \[\#\mathscr{I}(\text{mul}=m,\text{ gen}=k) \leq \#\mathscr{I}(\text{mul}=m,\text{ gen}=k+1). \]
\end{conjecture}

\begin{conjecture}
  If $k\geq m-1$ and $m\geq 6$, then
  \[\#\mathscr{L}(\text{mul}=m,\text{ gen}=k) \leq \#\mathscr{L}(\text{mul}=m,\text{ gen}=k+1). \]
\end{conjecture}

   Performing computations by applying the Algorithm  \ref{alg:fixedmultiplicity}, we attain the Table \ref{tab:fixedmultiplicity-genus-11and26}, where $ins_{gm}$ denotes the number of internal numerical semigroups with
genus $g$ and multiplicity $m$, $lns_{gm}$ the number of leaf numerical semigroups with genus $g$ and multiplicity $m$, and $ns_{gm}$ is the number of numerical semigroups with genus $g$ and multiplicity $m$, is to say, $ns_{gm} = ins_{gm} + lns_{gm}$. Besides, $ins_g$ symbolises the
number of internal numerical semigroups with genus $g$, $lns_g$ the
number of leaf numerical semigroups with genus $g$ and $ns_g$ is the
total number of numerical semigroups with genus $g$.

\begin{center}
  \begin{longtable}{|c c c c c | c c c|}
\caption{\label{tab:fixedmultiplicity-genus-11and26}  Computational results on the number of internal numerical semigroups for genera 11 and 26 for all available multiplicities, where $g$ denotes the genus of the numerical semigroup, $m$ its multiplicity,  $ins_{gm}$ is the number of internal numerical semigroups with genus $g$ and multiplicity $m$,  $lns_{gm}$ the number of leaf numerical semigroups with genus $g$ and multiplicity $m$,  $ns_{gm}$ the total number of numerical semigroups with genus $g$ and multiplicity $m$,  $ins_{g}$ the number of internal numerical semigroups with genus $g$,  $lns_{g}$ the number of leaf numerical semigroups with genus $g$ and $ns_g$ is the total number of numerical semigroups with genus $g$.}    \\

\hline
  \multicolumn{1}{|c}{$g$} & \multicolumn{1}{c}{$m$} & \multicolumn{1}{c}{$ins_{gm}$} & \multicolumn{1}{c}{$lns_{gm}$} & \multicolumn{1}{c|}{$ns_{gm}$} & \multicolumn{1}{c}{$ins_g$} & \multicolumn{1}{c}{$lns_g$} & \multicolumn{1}{c|}{$ns_g$} \\ \hline
  \endfirsthead

\hline \multicolumn{8}{|c|}%
{{\bfseries \tablename\ \thetable{} -- continued from previous page}} \\
\hline
  \multicolumn{1}{|c}{$g$} & \multicolumn{1}{c}{$m$} & \multicolumn{1}{c}{$ins_{gm}$} & \multicolumn{1}{c}{$lns_{gm}$} & \multicolumn{1}{c|}{$ns_{gm}$} & \multicolumn{1}{c}{$ins_g$} & \multicolumn{1}{c}{$lns_g$} & \multicolumn{1}{c|}{$ns_g$} \\ \hline
\endhead

\hline \multicolumn{8}{|r|}{\textbf{Continued on next page}} \\ \hline
\endfoot

\hline
\endlastfoot
  
   & 2 & 1 & 0 & 1 &  &  &   \\
   & 3 & 4 & 0 & 4 &  &  &   \\
   & 4 & 11& 4 & 15&  &  &   \\
   & 5 & 18& 6 & 24&  &  &   \\
   & 6 & 31& 18& 49&  &  &   \\
11 & 7 & 38& 26& 64& 233 & 110 & 343  \\
   & 8 & 47& 25& 72&  &  &   \\
   & 9 & 44& 22& 66&  &  &   \\
   & 10& 29& 8 & 37&  &  &   \\
   & 11& 9 & 1 & 10&  &  &   \\
   & 12& 1 & 0 & 1 &  &  &   \\ \hline
   &2 &1    &0    &1    &  &  &   \\ 
   &3 &9    &0    &9    &  &  &   \\ 
   &4 &58   &11   &69   &  &  &   \\ 
   &5 &156  &45   &201  &  &  &   \\ 
   &6 &555  &238  &793  &  &  &   \\ 
   &7 &998  &530  &1528 &  &  &   \\ 
   &8 &2473 &1562 &4035 &  &  &   \\ 
   &9 &4014 &2769 &6783 &  &  &   \\ 
   &10&7739 &5543 &13282&  &  &   \\ 
   &11&10933&8065 &18998&  &  &   \\ 
   &12&18892&13901&32793&  &  &   \\ 
   &13&23841&17470&41311&  &  &   \\ 
26 &14&35223&25177&60400&482133&288699&770832 \\ 
   &15&42927&30082&73009&  &  &   \\ 
   &16&52630&34759&87389&  &  &   \\ 
   &17&57383&36571&93954&  &  &   \\ 
   &18&59427&36261&95688&  &  &   \\ 
   &19&57484&32408&89892&  &  &   \\ 
   &20&48870&23640&72510&  &  &   \\ 
   &21&33370&12981&46351&  &  &   \\ 
   &22&17068&5077 &22145&  &  &   \\ 
   &23&6239 &1353 &7592 &  &  &   \\ 
   &24&1564 &232  &1796 &  &  &   \\ 
   &25&254  &23   &277  &  &  &   \\ 
   &26&24   &1    &25   &  &  &   \\ 
   &27&1    &0    &1    &  &  &   \\ 
\hline
\end{longtable}
\end{center}

From the data collected in tables \ref{tab:fixedmultiplicity-m5m8} and \ref{tab:fixedmultiplicity-genus-11and26}, we can establish  a new conjecture:

\begin{conjecture}
If $k\geq m-1$, then  $\#\mathscr{L}(\text{mul}=m,\text{ gen}=k)\leq \#\mathscr{I}(\text{mul}=m,\text{ gen}=k)$.
\end{conjecture}

\section{Internal numerical semigroups with fixed multiplicity and Frobenius number}
\label{sec:fixedmulFrob}

Along this section, we will denote as $m$ an integer greater or equal to two and $F$ an integer such that $m-1\leq F$ and $m\not| F$. Let denote by $\mathscr{I}(\text{mul}=m,\text{ Frob}=F)=\{S\in \mathscr{I} \text{ such that } m(S)=m \text{ and } F(S)=F\}$ and $\mathscr{L}(\text{mul}=m,\text{ Frob}=F)=\{S\in \mathscr{L} \text{ such that } m(S)=m \text{ and } F(S)=F\}$.

Let $S$ be a numerical semigroup such that $S\not = \mathbb{N}$. The \emph{ratio} of $S$ is $r(S)=\min\{s\in S \text{ such that } m(S)\not | s\}$. Note that $r(S)=\min\left(\text{msg}(S)\backslash \{m(S)\}\right)$.

A ratio-covariety is a family $\mathscr{R}$ of numerical semigroups satisfying the following properties:
\begin{enumerate}
\item $\mathscr{R}$ has a minimum.
\item If $\{S, T\}\subseteq\mathscr{R}$, then $S\cap T\in\mathscr{R}$.
\item If $S\in\mathscr{R}$ and $S\not = \min(\mathscr{R})$, then $S\backslash\{r(S)\}\in\mathscr{R}$.
\end{enumerate}

Now we can state the next result and provide a proof for it.

\begin{proposition}\label{Prop:31}
  $\mathscr{I}(\text{mul}=m,\text{ Frob}=F)$ is a ratio-covariety.
\end{proposition}

\begin{proof}
\;
  
  \begin{itemize}
  \item It is clear that $D(m, F)=\langle m \rangle \cup\{F+1,\to\}$ is the minimum of  $\mathscr{I}(\text{mul}=m,\text{ Frob}=F)$.
  \item If $\{S, T\}\subseteq \mathscr{I}(\text{mul}=m,\text{ Frob}=F)$, then by using the Proposition \ref{Prop:10} we know that $S\cap T\in\mathscr{I}$. Besides, it is obvious that $m(S\cap T)=m$ and $F(S\cap T)=F$. Therefore, $S\cap T\in \mathscr{I}(\text{mul}=m,\text{ Frob}=F)$.
    \item If $S\in \mathscr{I}(\text{mul}=m,\text{ Frob}=F)$ and $S\not = D(m, F)$, then $m< r(S) < F$ and by applying the Lemma \ref{Lem:1} we get that $S\backslash\{r(S)\}$ is a numerical semigroup with multiplicity $m$ and Frobenius number $F$. To conclude this proof, we will see that $S\backslash \{r(S)\}\in\mathscr{I}$. Indeed, as $S\in\mathscr{I}$ then there exists $x\in\text{msg}(S)$ such that $x> F$. It is clear that $x\in\text{msg}(S\backslash\{r(S)\})$ and thus $\mu(S\backslash\{r(S)\})\not = 0$. As a result $S\backslash\{r(S)\}\in\mathscr{I}(\text{mul}=m,\text{ Frob}=F)$. 
  \end{itemize}
\end{proof}

Let define the graph $G(\mathscr{I}(\text{mul}=m,\text{ Frob}=F))$ as follows:
$\mathscr{I}(\text{mul}=m,\text{ Frob}=F)$ is its set of vertexes and $(S,T)\in\mathscr{I}(\text{mul}=m,\text{ Frob}=F)\times \mathscr{I}(\text{mul}=m,\text{ Frob}=F)$ is a edge if and only if $T=S\backslash\{r(S)\}$. 

By applying the previous Proposition \ref{Prop:31} and the [Propositions 3 and 4 from \cite{Ratio}], we obtain the following result:

\begin{theorem}
  $G(\mathscr{I}(\text{mul}=m,\text{ Frob}=F))$ is a tree and $D(m,F)$ is its root. Besides, the set containing all the children of a vertex $S$ of the tree $G(\mathscr{I}(\text{mul}=m,\text{ Frob}=F))$ is
  \[\{S\cup\{x\} \text{ such that } x\in SG(S), \; m<x<r(S), \; x\not = F \text{ and } \mu(S\cup\{x\})\not = 0\} . \]
\end{theorem}

We have already introduced all the needed concepts to provide an algorithm to find all the internal numerical semigroups with multiplicity $m$ and Frobenius number $F$.

\begin{algorithm} \label{alg:fixedmultiplicity-Frobenius}
  Computation of $\mathscr{I}(\text{mul}=m,\text{ Frob}=F)$, i.e., all internal numerical semigroups with multiplicity $m$ and Frobenius number $F$. \\
  \; \\
\textbf{Input:} Two integer numbers, $m$ and $F$, such that $2\leq m\leq F+1$ and $m\not | F$.  \\
\textbf{Output:}  $\mathscr{I}(\text{mul}=m,\text{ Frob}=F)$. 
\begin{enumerate}[(1)]
\item  Compute $Ap(D(m,F),m)$.
\item  $A=\{D(m,F)\}$ and $B=\{D(m,F)\}$. 
\item  \textbf{for all} {$S\in B$} \textbf{do} 
\item  \hspace{0.5cm} Compute $\lambda(S)=\{x\in SG(S) \text{ such that } m<x<r(S), \; x\not =F$ \text{ and } $\mu(S\cup\{x\})\not = 0\}$. 
\item  \textbf{end for} 
\item  \textbf{if} {$ \bigcup\limits_{S\in B}\lambda(S)=\emptyset$} \textbf{then return} $A$. 
\item  \textbf{end if} 
\item  $C= \bigcup\limits_{S\in B} \{S\cup\{x\} \text{ such that }x\in\lambda(S)\}$. 
\item  $A=A\cup C$ and $B=C$. 
\item  \textbf{for all} {$S\in B$} \textbf{do} 
\item  \hspace{0.5cm} Compute $Ap(S,m)$. 
\item  \textbf{end for} 
\item  \textbf{go to} 3.
  \end{enumerate}
\end{algorithm}

Let see an example of usage of Algorithm \ref{alg:fixedmultiplicity-Frobenius} to compute all the internal numerical semigroups with multiplicity equals to 4 and Frobenius number equals to 10.

\begin{example}
  We are going to compute  $\mathscr{I}(\text{mul}=4,\text{ Frob}=10)$ by applying the Algorithm \ref{alg:fixedmultiplicity-Frobenius}.
  \begin{itemize}
  \item $Ap(D(4,10),4)=\{0,11,13,14\}$.
  \item $A=\{D(4,10)\}$ and $B=\{D(4,10)\}$.
  \item $\lambda(D(4,10))=\{7, 9\}$.
  \item $C=\{D(4,10)\cup\{7\}, D(4,10)\cup\{9\}\}$.
  \item $A=\{D(4,10), D(4,10)\cup\{7\}, D(4,10)\cup\{9\}\}$ and  $B=\{D(4,10)\cup\{7\}, D(4,10)\cup\{9\}\}$.
  \item $Ap(D(4,10)\cup\{7\},4)$=$\{0,7,13,14\}$ and $Ap(D(4,10)\cup\{9\},4)$=$\{0,9,11,14\}$.
  \item $\lambda(D(4,10)\cup\{7\})=\emptyset$ and  $\lambda(D(4,10)\cup\{9\})=\emptyset$.
  \item So the Algorithm  \ref{alg:fixedmultiplicity-Frobenius} finishes here and it returns
    \[\mathscr{I}(\text{mul}=4,\text{ Frob}=10)=\{D(4,10), D(4,10)\cup\{7\}, D(4,10)\cup\{9\}\}.\]    
  \end{itemize}
\end{example}

   Performing computations by applying the Algorithm  \ref{alg:fixedmultiplicity-Frobenius}, we attain the tables \ref{tab:fixedmultiplicity-F14} and \ref{tab:fixedFrobenius-m5}, where $ins_{Fm}$ denotes the number of internal numerical semigroups with
Frobenius number $F$ and multiplicity $m$, $lns_{Fm}$ the number of leaf numerical semigroups with Frobenius
number $F$ and multiplicity $m$, and $ns_{Fm}$ is the number of numerical semigroups with Frobenius number $F$ and multiplicity $m$, is to say, $ns_{Fm} = ins_{Fm} + lns_{Fm}$.

\begin{table}[H]
\begin{center}
  \begin{tabular}{|c c c c | c c c|}
    \hline
  \multicolumn{7}{ |c| }{Numerical semigroups with Frobenius number equals to 14} \\ \hline  
$m$ &$ins_{Fm}$ & $lns_{Fm}$ & $ns_{Fm}$ & $ins_F$ & $lns_F$ & $ns_F$ \\ \hline   
2  & 0 & 0 & 0 &  &  &   \\
3  &3  &0  &3  &  &  &   \\
4  &5  &1  &6  &  &  &   \\
5  &12 &2  &14 &  &  &   \\
6  &14 &2  &16 &  &  &   \\
7  & 0 & 0 & 0 &  &  &   \\
8  &32 &0  &32 &  &  &   \\
9  &16 &0  &16 &98&5 &103\\
10 &8  &0  &8  &  &  &   \\
11 &4  &0  &4  &  &  &   \\
12 &2  &0  &2  &  &  &   \\
13 &1  &0  &1  &  &  &   \\
14  & 0 & 0 & 0 &  &  &   \\
15 &1  &0  &1  &  &  &   \\ \hline
\end{tabular}
\caption{Computational results on the number of internal numerical semigroups for Frobenius number equals to 14 for all available multiplicities, where $F$ denotes the Frobenius number of the numerical semigroup, $m$ its multiplicity,  $ins_{Fm}$ is the number of internal numerical semigroups with Frobenius number $F$ and multiplicity $m$,  $lns_{Fm}$ the number of leaf numerical semigroups with Frobenius number $F$ and multiplicity $m$,  $ns_{Fm}$ the total number of numerical semigroups with Frobenius number $F$ and multiplicity $m$,  $ins_{F}$ the number of internal numerical semigroups with Frobenius numbers $F$,  $lns_{F}$ the number of leaf numerical semigroups with Frobenius number $F$ and $ns_F$ is the total number of numerical semigroups with Frobenius number $F$.}
\label{tab:fixedmultiplicity-F14}
\end{center}
\end{table}

\begin{remark}
  Algorithm  \ref{alg:fixedmultiplicity-Frobenius} is constructing the tree  $G(\mathscr{I}(\text{mul}=m,\text{ Frob}=F))$, (indeed all the algorithms provided are building a tree), so looking into the Table \ref{tab:fixedmultiplicity-F14} one can see that  $G(\mathscr{I}(\text{mul}=15,\text{ Frob}=14))$ is just a node, the root, so this question arises: In a tree with just one node, is the root an internal node or a leaf node?
  In this paper, when we speak about an internal node (respectively a leaf node) we are looking at the tree $G(\mathscr{S})$ formed by all numerical semigroups. It is to say, we are using the definitions provided in Section \ref{sec:intro} that state that an internal numerical semigroup is a numerical semigroup $S$ with $\mu(S)\not = 0$, respectively a leaf numerical semigroup has $\mu(S) = 0$.
\end{remark}

\begin{table}[H]
\begin{center}
\begin{tabular}{|c c c c ||c c c c|}
\hline  
  \multicolumn{8}{ |c| }{Numerical semigroups with multiplicity equals to 5} \\ 
\hline
  $F$ & $ins_{Fm}$ & $lns_{Fm}$ & $ns_{Fm}$ & $F$ & $ins_{Fm}$ & $lns_{Fm}$ & $ns_{Fm}$  \\ \hline
 
  4 & 1 & 0 & 1 &29&57&20&77\\ 
  5 & 0 & 0 & 0 &30&0 &0 &0  \\
  6 & 1 & 0 & 1 &31&57&21&78 \\
  7 & 2 & 0 & 2 &32&69&9 &78  \\
  8 & 4 & 0 & 4 &33&73&24&97  \\
  9 & 5 & 3 & 8 &34&88&10 &98\\
 10 & 0 & 0 & 0 &35&0 &0 &0  \\
11  &6  & 2 &8  &36&89&9 &98 \\
12  &6  &2  &8  &37&96&30&126  \\
13  &9  &5  &14 &38&112&10&122  \\
14  &12 &2  &14 &39&119&33&152 \\
15  &0  &0  &0  &40&0 &0 &0  \\ 
16  &11 &3  &14 &41&121&35&156 \\
17  &15 &7  &22 &42&141&13 &154 \\
18  &18 &2  &20 &43&145&42 &187 \\
19  &21 &10 &31 &44&173&14 &187  \\
20  &0  &0  &0  &45&0  &0 & 0 \\
21  &22 &9  &31 &46&172&17&189  \\
22  &26 &6  &32 &47&183&46&229  \\
23  &30 &13 &43 &48&208&14&222   \\
24  &38  &4 &42 &49&216&54&270  \\
25  &0   &0 &0  &50&0  &0 &0  \\ 
26  &38  &6 &44 &51&219&52 &271 \\
27  &43  &15&58 &52&251&22 &273  \\
28  &51  &6 &57 &  &  & &  \\
\hline
\end{tabular}
\caption{Computational results on the number of numerical semigroups with multiplicity equals to $5$ up to Frobenius number $52$, where $F$ denotes the Frobenius number of the numerical semigroup, $m$ its multiplicity,  $ins_{Fm}$ is the number of internal numerical semigroups with Frobenius number $F$ and multiplicity $m$,  $lns_{Fm}$ the number of leaf numerical semigroups with Frobenius number $F$ and multiplicity $m$ and $ns_{Fm}$ the total number of numerical semigroups with Frobenius number $F$ and multiplicity $m$.}
\label{tab:fixedFrobenius-m5}
\end{center}
\end{table}

From the information depicted in the tables \ref{tab:fixedmultiplicity-F14} and \ref{tab:fixedFrobenius-m5} one can point out the following conjecture:

\begin{conjecture}
Let $m$ and $F$ be two integer numbers such that  $2\leq m\leq F+1$ and $m\not | F$, then  $\#\mathscr{L}(\text{mul}=m,\text{ Frob}=F)\leq \#\mathscr{I}(\text{mul}=m,\text{ Frob}=F)$.
\end{conjecture}

Also, from the Table  \ref{tab:fixedmultiplicity-F14} we can establish the next result:
\begin{proposition}\label{Prop:34}
  Let $F$ be an integer number such that  $F\geq 4$ and $\left\lceil \frac{F}{2} + 1 \right\rceil \leq m < F$, then
  \[ \#\mathscr{I}(\text{mul}=m,\text{ Frob}=F) = 2^{F-m-1}, \]
  and,
  \[\#\mathscr{L}(\text{mul}=m,\text{ Frob}=F)=0.\]

  Where $\left\lceil q \right\rceil$ is the ceiling function of $q$, which returns the least integer greater than or equal to the rational number $q$.
\end{proposition}
\begin{proof}
  Let $F$ be an integer number such that  $F\geq 4$ and  $\left\lceil \frac{F}{2} + 1 \right\rceil \leq m < F$. Then, the smallest numerical semigroup S with multiplicity $m$ and Frobenius number $F$ has the form (as $m>\frac{F}{2}$)
  \[S=\{0,m,F+1,F+2,\to\}.\]

  So all the numerical semigroups with multiplicity $m$ and Frobenius number $F$ can be obtained from $S$ by adding elements from $A=\{m+1, m+2,\ldots, F-2, F-1\}$. Also, all this numerical semigroups are internal as $F+1$ belongs always to the minimal system of generators of the semigroup, so $\mu\not=0$. As $\#\mathcal{P}(A)=2^{F-1-(m+1)+1}=2^{F-m-1}$, therefore we get that $\#\mathscr{I}(\text{mul}=m,\text{ Frob}=F) = 2^{F-m-1}$.
So, in this case where all the numerical semigroups are internal, we get that $\#\mathscr{L}(\text{mul}=m,\text{ Frob}=F)=0$.
\end{proof}

As a direct consequence of Proposition \ref{Prop:34} we get the next corollary:
\begin{corollary}
  Let $F$ be an integer number such that  $F\geq 4$, $\left\lceil \frac{F}{2}\right\rceil \leq m < F$ and  $m\not | F$, then the total number of numerical semigroups with multiplicity $m$ and Frobenius number F is
  \[ ns_{Fm} = 2^{F-m-1}. \]
\end{corollary}

\section{Conclusions}
In this work, we have studied the tree of  numerical semigroups, by classifying a numerical semigroup as internal or leaf, depending whether the numerical semigroup is an internal or a leaf node in the tree.
To perform this analysis several algorithms have been provided to compute the internal numerical semigroups by keeping fixed one of the three invariants, is to say, multiplicity, Frobenius number or genus. In all the three cases, we have implemented the algorithms in C++ to compute tables which has allowed us to study the behaviour of the number of numerical semigroups with the corresponding invariant fixed. We have observed that in all the situations there are more internal than leaf numerical semigroups.

Besides, in the last section we have fixed multiplicity and Frobenius number simultaneously  and the results obtained are analogous than fixing just one invariant, but in this last section we have found a closed formula to count the exact number of numerical semigroups for some values of Frobenius number and multiplicity. Indeed we have obtained that in those cases all the numerical semigroups are internal nodes of the tree.

\begin{sloppypar} 
The source code of the implementation of the algorithms in C++ is available at the next URL: 

\url{https://drive.google.com/drive/folders/1cF_bWf94JGJW2p24mN1Mcl7ACmWTljIa?usp=sharing} \; \; .
\end{sloppypar} 

\section*{Acknowledgement}

This work was supported by a Departmental Collaboration Grant from the Spanish Ministry of Education, Vocational Training and Sports (2025/2026 call). The author Mario Casas gratefully acknowledge the financial support provided by this Departmental Collaboration Grant.

\bibliographystyle{unsrt}
\bibliography{../../Referencias}
\clearpage

\pagenumbering{roman}
\appendix

\section{Table for internal numerical semigroups with fixed multiplicity}

   Performing computations by applying the Algorithm  \ref{alg:fixedmultiplicity}, we attain the Table \ref{tab:fixedmultiplicity}, where $ins_{gm}$ denotes the number of internal numerical semigroups with
genus $g$ and multiplicity $m$, $lns_{gm}$ the number of leaf numerical semigroups with genus $g$ and multiplicity $m$, and $ns_{gm}$ is the number of numerical semigroups with genus $g$ and multiplicity $m$, is to say, $ns_{gm} = ins_{gm} + lns_{gm}$. Besides, $ins_g$ symbolises the
number of internal numerical semigroups with genus $g$, $lns_g$ the
number of leaf numerical semigroups with genus $g$ and $ns_g$ is the
total number of numerical semigroups with genus $g$.

\begin{center}
  \begin{longtable}{|c c c c c | c c c|}
\caption{\label{tab:fixedmultiplicity}  Computational results on the number of internal numerical semigroups up to genus 26 for all available multiplicities, where $g$ denotes the genus of the numerical semigroup, $m$ its multiplicity,  $ins_{gm}$ is the number of internal numerical semigroups with genus $g$ and multiplicity $m$,  $lns_{gm}$ the number of leaf numerical semigroups with genus $g$ and multiplicity $m$,  $ns_{gm}$ the total number of numerical semigroups with genus $g$ and multiplicity $m$,  $ins_{g}$ the number of internal numerical semigroups with genus $g$,  $lns_{g}$ the number of leaf numerical semigroups with genus $g$ and $ns_g$ is the total number of numerical semigroups with genus $g$.}    \\

\hline
  \multicolumn{1}{|c}{$g$} & \multicolumn{1}{c}{$m$} & \multicolumn{1}{c}{$ins_{gm}$} & \multicolumn{1}{c}{$lns_{gm}$} & \multicolumn{1}{c|}{$ns_{gm}$} & \multicolumn{1}{c}{$ins_g$} & \multicolumn{1}{c}{$lns_g$} & \multicolumn{1}{c|}{$ns_g$} \\ \hline
  \endfirsthead

\hline \multicolumn{8}{|c|}%
{{\bfseries \tablename\ \thetable{} -- continued from previous page}} \\
\hline
  \multicolumn{1}{|c}{$g$} & \multicolumn{1}{c}{$m$} & \multicolumn{1}{c}{$ins_{gm}$} & \multicolumn{1}{c}{$lns_{gm}$} & \multicolumn{1}{c|}{$ns_{gm}$} & \multicolumn{1}{c}{$ins_g$} & \multicolumn{1}{c}{$lns_g$} & \multicolumn{1}{c|}{$ns_g$} \\ \hline
\endhead

\hline \multicolumn{8}{|r|}{\textbf{Continued on next page}} \\ \hline
\endfoot

\hline
\endlastfoot
  
  0 & 1 & 1 & 0 & 1 & 1 & 0 & 1\\ \hline
  1 & 2 & 1 & 0 & 1 & 1 & 0 & 1 \\ \hline
    & 2 & 1 & 0 & 1 &  &  &  \\ 
  2 & 3 & 1 & 0 & 1 & 2 & 0 & 2 \\ \hline
    & 2 & 1 & 0 & 1 &  &  &  \\
  3 & 3 & 1 & 1 & 2 & 3 & 1 & 4 \\
    & 4 & 1 & 0 & 1 &  &  &  \\ \hline
    & 2 & 1 & 0 & 1 &  &  &  \\
  4 & 3 & 1 & 1 & 2 & 5 & 2 & 7  \\
    & 4 & 2 & 1 & 3 &  &  &  \\
    & 5 & 1 & 0 & 1 &  &  &  \\ \hline
    & 2 & 1 & 0 & 1 &  &  &   \\
    & 3 & 2 & 0 & 2 &  &  &   \\
  5 & 4 & 3 & 1 & 4 & 10 & 2 & 12  \\
    & 5 & 3 & 1 & 4 &  &  &   \\
    & 6 & 1 & 0 & 1 &  &  &   \\ \hline
    & 2 & 1 & 0 & 1 &  &  &   \\
    & 3 & 2 & 1 & 3 &  &  &   \\
  6 & 4 & 3 & 3 & 6 & 15 & 8  & 23  \\
    & 5 & 4 & 3 & 7 &  &  &   \\
    & 6 & 4 & 1 & 5 &  &  &   \\
    & 7 & 1 & 0 & 1 &  &  &   \\ \hline
   & 2 & 1 & 0 & 1 &  &  &   \\
   & 3 & 2 & 1 & 3 &  &  &   \\
   & 4 & 5 & 2 & 7 &  &  &   \\
  7 & 5 & 6 & 4 & 10 & 27 & 12 & 39  \\
    & 6 & 7 & 4 & 11 &  &  &   \\
    & 7 & 5 & 1 & 6 &  &  &   \\
    & 8 & 1 & 0 & 1 &  &  &   \\ \hline
    & 2 & 1 & 0 & 1 &  &  &   \\
    & 3 & 3 & 0 & 3 &  &  &   \\
    & 4 & 6 & 3 & 9 &  &  &   \\
  8 & 5 & 7 & 6 & 13 & 47 & 20 & 67  \\
    & 6 & 12& 5 & 17&  &  &   \\
    & 7 & 11& 5 & 16 &  &  &   \\
    & 8 & 6 & 1 & 7 &  &  &   \\
    & 9 & 1 & 0 & 1 &  &  &   \\ \hline
   & 2  & 1 & 0 & 1 &  &  &   \\
   & 3  & 3 & 1 & 4 &  &  &   \\
   & 4  & 8 & 3 & 11&  &  &   \\
   & 5  & 12& 4 & 16&  &  &   \\
 9 & 6  & 16& 11& 27& 81 & 37  & 118  \\
   & 7  & 17& 11& 28&  &  &   \\
   & 8  & 16& 6 & 22&  &  &   \\
   & 9  & 7 & 1 & 8 &  &  &   \\
   & 10 & 1 & 0 & 1 &  &  &   \\ \hline
   & 2 & 1 & 0 & 1 &  &  &   \\
   & 3 & 3 & 1 & 4 &  &  &   \\
   & 4 & 9 & 4 & 13&  &  &   \\
   & 5 & 11& 11& 22&  &  &   \\
10 & 6 & 22& 15& 37& 132 & 72 & 204  \\
   & 7 & 27& 17& 44&  &  &   \\
   & 8 & 28& 16& 44&  &  &   \\
   & 9 & 22&  7& 29&  &  &   \\
   & 10& 8 & 1 & 9 &  &  &   \\
   & 11& 1 & 0 & 1 &  &  &   \\ \hline
   & 2 & 1 & 0 & 1 &  &  &   \\
   & 3 & 4 & 0 & 4 &  &  &   \\
   & 4 & 11& 4 & 15&  &  &   \\
   & 5 & 18& 6 & 24&  &  &   \\
   & 6 & 31& 18& 49&  &  &   \\
11 & 7 & 38& 26& 64& 233 & 110 & 343  \\
   & 8 & 47& 25& 72&  &  &   \\
   & 9 & 44& 22& 66&  &  &   \\
   & 10& 29& 8 & 37&  &  &   \\
   & 11& 9 & 1 & 10&  &  &   \\
   & 12& 1 & 0 & 1 &  &  &   \\ \hline
   & 2 & 1 & 0 & 1 &  &  &   \\
   & 3 & 4 & 1 & 5 &  &  &   \\
   & 4 & 13& 5 & 18&  &  &   \\
   & 5 & 18& 14& 32&  &  &   \\
   & 6 & 39& 27& 66&  &  &   \\
12 & 7 & 49& 36& 85& 381 & 211 & 592   \\
   & 8 & 71& 45&116&  &  &   \\
   & 9 & 72& 44&116&  &  &   \\
   & 10& 66& 29& 95&  &  &   \\
   & 11& 37& 9 & 46&  &  &   \\ 
   & 12& 10& 1 & 11&  &  &   \\ 
   & 13& 1 & 0 & 1 &  &  &   \\ \hline 
   & 2 & 1 & 0 & 1 &  &  &   \\ 
   & 3 & 4 & 1 & 5  &  &  &   \\ 
   & 4 & 15& 5 & 20 &  &  &   \\ 
   & 5 & 24& 11& 35 &  &  &   \\ 
   & 6 & 53& 32& 85 &  &  &   \\ 
   & 7 & 66& 46& 112&  &  &   \\ 
13 & 8 &102& 70& 172& 651 & 350 & 1001  \\ 
   & 9 &117& 71& 188&  &  &   \\ 
   & 10&116& 66& 182&  &  &   \\ 
   & 11& 95& 37& 132&  &  &   \\ 
   & 12& 46& 10& 56 &  &  &   \\ 
   & 13& 11& 1 & 12 &  &  &   \\ 
   & 14&  1& 0 & 1  &  &  &   \\ \hline 
   & 2 & 1 & 0 & 1 &  &  &   \\ 
   & 3 & 5 & 0 & 5 &  &  &   \\ 
   & 4 &18 & 5 & 23&  &  &   \\ 
   & 5 &30 &13 & 43&  &  &   \\ 
   & 6 &68 &38 &106&  &  &   \\ 
   & 7 &90 &58 &148&  &  &   \\ 
14 & 8 &141&98 &239&1103 & 590 & 1693  \\ 
   & 9 &173&115&288&  &  &   \\ 
   & 10&194&110&304&  &  &   \\ 
   & 11&182&95 &277&  &  &   \\ 
   & 12&132&46 &178&  &  &   \\ 
   & 13&56 &11 & 67&  &  &   \\ 
   & 14&12 & 1 & 13&  &  &   \\ 
   & 15&  1& 0 & 1 &  &  &   \\ \hline
   & 2 & 1 & 0 & 1 &  &  &   \\ 
   & 3 & 5 & 1 & 6 &  &  &   \\ 
   & 4 &20 & 6 &26  &  &  &   \\ 
   & 5 & 37&14 &51  &  &  &   \\ 
   & 6 &83 &50 &133 &  &  &   \\ 
   & 7 &109&82 &191 &  &  &   \\ 
   & 8 &198&127&325 &  &  &   \\ 
15 & 9 &240&169&409 &1836 & 1021 & 2857  \\ 
   &10 &303&189&492 &  &  &   \\ 
   &11 &304&182&486 &  &  &   \\ 
   &12 &277&132&409 &  &  &   \\ 
   &13 &178&56 &234 &  &  &   \\ 
   &14 &67 &12 &79  &  &  &   \\ 
   &15 &13 & 1 &14  &  &  &   \\ 
   &16 &1  & 0 &1   &  &  &   \\ \hline
   & 2 &1  &0  &1  &  &  &   \\ 
   & 3 &5  &1  &6  &  &  &   \\ 
   & 4 &22 &7  &29 &  &  &   \\ 
   & 5 &40 &21 &61 &  &  &   \\ 
   & 6 &101&62 &163&  &  &   \\ 
   & 7 &144&93 &237&  &  &   \\ 
   & 8 &256&185&441&  &  &   \\ 
16 & 9 &322&237&559&3064 &1742  & 4806  \\ 
   &10 &455&299&754&  &  &   \\ 
   &11 &495&301&796&  &  &   \\ 
   &12 &486&277&763&  &  &   \\ 
   &13 &409&178&587&  &  &   \\ 
   &14 &234&67 &301&  &  &   \\ 
   &15 &79 &13 &92 &  &  &   \\ 
   &16 &14 &1  &15 &  &  &   \\ 
   &17 &1  &0  & 1 &  &  &   \\ \hline
   &2  &1  &0  &1   &  &  &   \\ 
   &3  &6  &0  &6   &  &  &   \\ 
   &4  &26 &6  &32  &  &  &   \\ 
   &5  &50 &18 &68  &  &  &   \\ 
   &6  &126&70 &196 &  &  &   \\ 
   &7  &181&120&301 &  &  &   \\ 
   &8  &337&236&573 &  &  &   \\ 
   &9  &449&301&750 &  &  &   \\ 
17 &10 &647&447&1094&5144&2901& 8045  \\ 
   &11 &756&490&1246&  &  &   \\ 
   &12 &806&476&1282&  &  &   \\ 
   &13 &763&409&1172&  &  &   \\ 
   &14 &587&234&821 &  &  &   \\ 
   &15 &301&79 &380 &  &  &   \\ 
   &16 &92 &14 &106 &  &  &   \\ 
   &17 &15 &1  &16  &  &  &   \\ 
   &18 &1  &0  &1   &  &  &   \\ \hline
   &2 &1   &0  &1   &  &  &   \\ 
   &3 &6   &1  &7   &  &  &   \\ 
   &4 &28  &8  &36  &  &  &   \\ 
   &5 &55  &25 &80  &  &  &   \\ 
   &6 &154 &82 &236 &  &  &   \\ 
   &7 &217 &152&369 &  &  &   \\ 
   &8 &436 &301&737 &  &  &   \\ 
   &9 &585 &430&1015&  &  &   \\ 
18 &10&904 &630&1534&8540&4927& 13467  \\ 
   &11&1094&747&1841&  &  &   \\ 
   &12&1282&792&2074&  &  &   \\ 
   &13&1282&763&2045&  &  &   \\ 
   &14&1172&587&1759&  &  &   \\ 
   &15&821 &301&1122&  &  &   \\ 
   &16&380 &92 &472 &  &  &   \\ 
   &17&106 &15 &121 &  &  &   \\ 
   &18&16  &1  &17  &  &  &   \\ 
   &19&1   &0  &1   &  &  &   \\ \hline
   &2  &1   &0   &1   &  &  &   \\ 
   &3  &6   &1   &7   &  &  &   \\ 
   &4  &32  &7   &39  &  &  &   \\ 
   &5  &67  &22  &89  &  &  &   \\ 
   &6  &181 &101 &282 &  &  &   \\ 
   &7  &273 &171 &444 &  &  &   \\ 
   &8  &564 &381 &945 &  &  &   \\ 
   &9  &777 &557 &1334&  &  &   \\ 
   &10 &1228&878 &2106&  &  &   \\ 
19 &11 &1516&1085&2601&14220&8244&22464 \\ 
   &12 &1957&1270&3227&  &  &   \\ 
   &13 &2081&1275&3356&  &  &   \\ 
   &14 &2045&1172&3217&  &  &   \\ 
   &15 &1759&821 &2580&  &  &   \\ 
   &16 &1122&380 &1502&  &  &   \\ 
   &17 &472 &106 &578 &  &  &   \\ 
   &18 &121 &16  &137 &  &  &   \\ 
   &19 &17  &1   &18  &  &  &   \\ 
   &20 & 1  &0   & 1  &  &  &   \\ \hline
   &2 &1   &0   &1   &  &  &   \\ 
   &3 &7   &0   &7   &  &  &   \\ 
   &4 &35  &8   &43  &  &  &   \\ 
   &5 &75  &29  &104 &  &  &   \\ 
   &6 &220 &110 &330 &  &  &   \\ 
   &7 &345 &196 &541 &  &  &   \\ 
   &8 &711 &482 &1193&  &  &   \\ 
   &9 &1007&730 &1737&  &  &   \\ 
   &10&1667&1173&2840&  &  &   \\ 
20 &11&2067&1494&3561&23645&13751&37396 \\ 
   &12&2872&1940&4812&  &  &   \\ 
   &13&3229&2072&5301&  &  &   \\ 
   &14&3376&2025&5401&  &  &   \\ 
   &15&3217&1759&4976&  &  &   \\ 
   &16&2580&1122&3702&  &  &   \\ 
   &17&1502&472 &1974&  &  &   \\ 
   &18&578 &121 &699 &  &  &   \\ 
   &19&137 &17  &154 &  &  &   \\ 
   &20&18  &1   &19  &  &  &   \\ 
   &21&1   &0   & 1  &  &  &   \\ \hline
   &2 &1   &0   &1   &  &  &   \\ 
   &3 &7   &1   &8   &  &  &   \\ 
   &4 &38  &9   &47  &  &  &   \\ 
   &5 &90  &25  &115 &  &  &   \\ 
   &6 &266 &124 &390 &  &  &   \\ 
   &7 &408 &250 &658 &  &  &   \\ 
   &8 &889 &601 &1490&  &  &   \\ 
   &9 &1313&918 &2231&  &  &   \\ 
   &10&2199&1594&3793&  &  &   \\ 
   &11&2805&2017&4822&  &  &   \\ 
21 &12&4099&2840&6939&39235&22959&62194 \\ 
   &13&4814&3206&8020&  &  &   \\ 
   &14&5380&3341&8721&  &  &   \\ 
   &15&5401&3217&8618&  &  &   \\ 
   &16&4976&2580&7556&  &  &   \\ 
   &17&3702&1502&5204&  &  &   \\ 
   &18&1974&578 &2552&  &  &   \\ 
   &19&699 &137 &836 &  &  &   \\ 
   &20&154 &18  &172  &  &  &   \\ 
   &21&19  &1   &20   &  &  &   \\ 
   &22&1   &0   &1    &  &  &   \\ \hline
   &2 &1   &0   &1    &  &  &   \\ 
   &3 &7   &1   &8    &  &  &   \\ 
   &4 &42  &9   &51   &  &  &   \\ 
   &5 &94  &39  &133  &  &  &   \\ 
   &6 &307 &149 &456  &  &  &   \\ 
   &7 &501 &283 &784  &  &  &   \\ 
   &8 &1094&753 &1847 &  &  &   \\ 
   &9 &1651&1200&2851 &  &  &   \\ 
   &10&2877&2090&4967 &  &  &   \\ 
   &11&3761&2729&6490 &  &  &   \\ 
22 &12&5702&4050&9752 &64890&38356&103246 \\ 
   &13&6865&4792&11657&  &  &   \\ 
   &14&8294&5349&13643&  &  &   \\ 
   &15&8735&5387&14122&  &  &   \\ 
   &16&8618&4976&13594&  &  &   \\ 
   &17&7556&3702&11258&  &  &   \\ 
   &18&5204&1974&7178 &  &  &   \\ 
   &19&2552&699 &3251 &  &  &   \\ 
   &20&836 &154 &990  &  &  &   \\ 
   &21&172 &19  &191  &  &  &   \\ 
   &22&20  &1   &21   &  &  &   \\ 
   &23&1   &0   &1    &  &  &   \\ \hline
   &2 &1    &0   &1    &  &  &   \\ 
   &3 &8    &0   &8    &  &  &   \\ 
   &4 &46   &9   &55   &  &  &   \\ 
   &5 &113  &30  &143  &  &  &   \\ 
   &6 &360  &165 &525  &  &  &   \\ 
   &7 &602  &338 &940  &  &  &   \\ 
   &8 &1365 &897 &2262 &  &  &   \\ 
   &9 &2100 &1482&3582 &  &  &   \\ 
   &10&3732 &2694&6426 &  &  &   \\ 
   &11&4981 &3657&8638 &  &  &   \\ 
   &12&7843 &5596&13439&  &  &   \\ 
23 &13&9557 &6818&16375&107403&63560&170963 \\ 
   &14&12394&8246&20640&  &  &   \\ 
   &15&13658&8706&22364&  &  &   \\ 
   &16&14159&8581&22740&  &  &   \\ 
   &17&13594&7556&21150&  &  &   \\ 
   &18&11258&5204&16462&  &  &   \\ 
   &19&7178 &2552&9730 &  &  &   \\ 
   &20&3251 &836 &4087 &  &  &   \\ 
   &21&990  &172 &1162 &  &  &   \\ 
   &22&191  &20  &211  &  &  &   \\ 
   &23&21   &1   &22   &  &  &   \\ 
   &24&1    &0   &1    &  &  &   \\ \hline
   &2 &1    &0    &1    &  &  &   \\ 
   &3 &8    &1    &9    &  &  &   \\ 
   &4 &50   &10   &60   &  &  &   \\ 
   &5 &124  &40   &164  &  &  &   \\ 
   &6 &419  &189  &608  &  &  &   \\ 
   &7 &713  &401  &1114 &  &  &   \\ 
   &8 &1681 &1085 &2766 &  &  &   \\ 
   &9 &2590 &1885 &4475 &  &  &   \\ 
   &10&4807 &3445 &8252 &  &  &   \\ 
   &11&6545 &4810 &11355&  &  &   \\ 
   &12&10653&7656 &18309&  &  &   \\ 
24 &13&13034&9484 &22518&177349&105479&282828\\ 
   &14&17968&12314&30282&  &  &   \\ 
   &15&20644&13608&34252&  &  &   \\ 
   &16&22547&14067&36614&  &  &   \\ 
   &17&22740&13594&36334&  &  &   \\ 
   &18&21150&11258&32408&  &  &   \\ 
   &19&16462&7178 &23640&  &  &   \\ 
   &20&9730 &3251 &12981&  &  &   \\ 
   &21&4087 &990  &5077 &  &  &   \\ 
   &22&1162 &191  &1353 &  &  &   \\ 
   &23&211  &21   &232  &  &  &   \\ 
   &24&22   &1    &23   &  &  &   \\ 
   &25&1    &0    &1    &  &  &   \\ \hline
   &2 &1    &0    &1    &  &  &   \\ 
   &3 &8    &1    &9    &  &  &   \\ 
   &4 &54   &10   &64   &  &  &   \\ 
   &5 &140  &41   &181  &  &  &   \\ 
   &6 &482  &215  &697  &  &  &   \\ 
   &7 &843  &464  &1307 &  &  &   \\ 
   &8 &2046 &1307 &3353 &  &  &   \\ 
   &9 &3245 &2267 &5512 &  &  &   \\ 
   &10&6124 &4395 &10519&  &  &   \\ 
   &11&8469 &6287 &14756&  &  &   \\ 
   &12&14261&10402&24663&  &  &   \\ 
   &13&17734&12854&30588&  &  &   \\ 
25 &14&25371&17863&43234&292551&174673&467224\\ 
   &15&30197&20575&50772&  &  &   \\ 
   &16&34880&22460&57340&  &  &   \\ 
   &17&36647&22707&59354&  &  &   \\ 
   &18&36334&21150&57484&  &  &   \\ 
   &19&32408&16462&48870&  &  &   \\ 
   &20&23640&9730 &33370&  &  &   \\ 
   &21&12981&4087 &17068&  &  &   \\ 
   &22&5077 &1162 &6239 &  &  &   \\ 
   &23&1353 &211  &1564 &  &  &   \\ 
   &24&232  &22   &254  &  &  &   \\ 
   &25&23   &1    &24   &  &  &   \\ 
   &26&1    &0    &1    &  &  &   \\ \hline
   &2 &1    &0    &1    &  &  &   \\ 
   &3 &9    &0    &9    &  &  &   \\ 
   &4 &58   &11   &69   &  &  &   \\ 
   &5 &156  &45   &201  &  &  &   \\ 
   &6 &555  &238  &793  &  &  &   \\ 
   &7 &998  &530  &1528 &  &  &   \\ 
   &8 &2473 &1562 &4035 &  &  &   \\ 
   &9 &4014 &2769 &6783 &  &  &   \\ 
   &10&7739 &5543 &13282&  &  &   \\ 
   &11&10933&8065 &18998&  &  &   \\ 
   &12&18892&13901&32793&  &  &   \\ 
   &13&23841&17470&41311&  &  &   \\ 
26 &14&35223&25177&60400&482133&288699&770832 \\ 
   &15&42927&30082&73009&  &  &   \\ 
   &16&52630&34759&87389&  &  &   \\ 
   &17&57383&36571&93954&  &  &   \\ 
   &18&59427&36261&95688&  &  &   \\ 
   &19&57484&32408&89892&  &  &   \\ 
   &20&48870&23640&72510&  &  &   \\ 
   &21&33370&12981&46351&  &  &   \\ 
   &22&17068&5077 &22145&  &  &   \\ 
   &23&6239 &1353 &7592 &  &  &   \\ 
   &24&1564 &232  &1796 &  &  &   \\ 
   &25&254  &23   &277  &  &  &   \\ 
   &26&24   &1    &25   &  &  &   \\ 
   &27&1    &0    &1    &  &  &   \\ 
\hline
\end{longtable}
\end{center}

\clearpage

\section{Table for internal numerical semigroups with fixed multiplicity and Frobenius number}

   Performing computations by applying the Algorithm  \ref{alg:fixedmultiplicity-Frobenius}, we attain the Table \ref{tab:fixedmultiplicity-Frobenius}, where $ins_{Fm}$ denotes the amount of internal numerical semigroups with
Frobenius number $F$ and multiplicity $m$, $lns_{Fm}$ the quantity of leaf numerical semigroups with Frobenius
number $F$ and multiplicity $m$, and $ns_{Fm}$ is the sum total of numerical semigroups with Frobenius number $F$ and multiplicity $m$, is to say, $ns_{Fm} = ins_{Fm} + lns_{Fm}$.

\begin{center}
  \begin{longtable}{|c c c c c | c c c|}
\caption{\label{tab:fixedmultiplicity-Frobenius}  Computational results on the amount of internal numerical semigroups up to Frobenius number 26 for all available multiplicities, where $F$ denotes the Frobenius number of the numerical semigroup, $m$ its multiplicity,  $ins_{Fm}$ is the quantity of internal numerical semigroups with Frobenius number $F$ and multiplicity $m$,  $lns_{Fm}$ the amount of leaf numerical semigroups with Frobenius number $F$ and multiplicity $m$,  $ns_{Fm}$ the sum total of numerical semigroups with Frobenius number $F$ and multiplicity $m$,  $ins_{F}$ the quantity of internal numerical semigroups with Frobenius numbers $F$,  $lns_{F}$ the count of leaf numerical semigroups with Frobenius number $F$ and $ns_F$ is the total of numerical semigroups with Frobenius number $F$.}    \\

\hline
  \multicolumn{1}{|c}{$F$} & \multicolumn{1}{c}{$m$} & \multicolumn{1}{c}{$ins_{Fm}$} & \multicolumn{1}{c}{$lns_{Fm}$} & \multicolumn{1}{c|}{$ns_{Fm}$} & \multicolumn{1}{c}{$ins_F$} & \multicolumn{1}{c}{$lns_F$} & \multicolumn{1}{c|}{$ns_F$} \\ \hline
  \endfirsthead

\hline \multicolumn{8}{|c|}%
{{\bfseries \tablename\ \thetable{} -- continued from previous page}} \\
\hline
  \multicolumn{1}{|c}{$F$} & \multicolumn{1}{c}{$m$} & \multicolumn{1}{c}{$ins_{Fm}$} & \multicolumn{1}{c}{$lns_{Fm}$} & \multicolumn{1}{c|}{$ns_{Fm}$} & \multicolumn{1}{c}{$ins_F$} & \multicolumn{1}{c}{$lns_F$} & \multicolumn{1}{c|}{$ns_F$} \\ \hline
\endhead

\hline \multicolumn{8}{|r|}{\textbf{Continued on next page}} \\ \hline
\endfoot

\hline
\endlastfoot
  
  1 & 2 & 1 & 0 & 1 & 1 & 0 & 1\\ \hline
  2 & 3 & 1 & 0 & 1 & 1 & 0 & 1 \\ \hline
  3 & 2 & 1 & 0 & 1 & 2 & 0 & 2 \\ 
    & 4 & 1 & 0 & 1 &   &   &  \\ \hline
  4 & 3 & 1 & 0 & 1 & 2 & 0 & 2 \\
    & 5 & 1 & 0 & 1 &  &  &  \\ \hline
    & 2 & 1 & 0 & 1 &  &  &   \\
  5 & 3 & 1 & 1 & 2 & 4&1 &5   \\
    & 4 & 1 & 0 & 1 &  &  &   \\
    & 6 & 1 & 0 & 1 &  &  &   \\ \hline
    & 4 & 2 & 0 & 2 &  &  &   \\
  6 & 5 & 1 & 0 & 1 &4 &0 &4  \\
    & 7 & 1 & 0 & 1 &  &  &   \\ \hline
    & 2 & 1 & 0 & 1 &  &  &   \\
    & 3 & 1 & 1 & 2 &  &  &   \\
  7 & 4 & 3 & 1 & 4 &9 &2 &11 \\
    & 5 & 2 & 0 & 2 &  &  &   \\
    & 6 & 1 & 0 & 1 &  &  &   \\
    & 8 & 1 & 0 & 1 &  &  &   \\ \hline
    & 3 & 2 & 0 & 2 &  &  &   \\
    & 5 & 4 & 0 & 4 &  &  &   \\
  8 & 6 & 2 & 0 & 2 &10&0 &10 \\
    & 7 & 1 & 0 & 1 &  &  &   \\
    & 9 & 1 & 0 & 1 &  &  &   \\ \hline
    & 2 & 1 & 0 & 1 &  &  &   \\
    & 4 & 3 & 1 & 4 &  &  &   \\
 9  & 5 & 5 & 3 & 8 &17&4 &21 \\
    & 6 & 4 & 0 & 4 &  &  &   \\
    & 7 & 2 & 0 & 2 &  &  &   \\
    & 8 & 1 & 0 & 1 &  &  &   \\
    & 10& 1 & 0 & 1 &  &  &   \\ \hline
    &3  & 2 & 0 & 2 &  &  &   \\
    &4  & 3 & 1 & 4 &  &  &   \\
    &6  & 8 & 0 & 8 &  &  &   \\
10  &7  & 4 & 0 & 4 &21&1 &22 \\
    &8  & 2 & 0 & 2 &  &  &   \\
    &9  & 1 & 0 & 1 &  &  &   \\
    &11 & 1 & 0 & 1 &  &  &   \\ \hline
    &2  &1  & 0 &1  &  &  &   \\
    &3  &2  & 1 &3  &  &  &   \\
    &4  &5  & 2 &7  &  &  &   \\
    &5  &6  & 2 &8  &  &  &   \\
11  &6  &12 & 4 &16 &42&9 &51 \\
    &7  &8  & 0 &8  &  &  &   \\
    &8  &4  & 0 &4  &  &  &   \\
    &9  &2  & 0 &2  &  &  &   \\
    &10 &1  & 0 &1  &  &  &   \\
    &12 &1  & 0 &1  &  &  &   \\\hline
    &5  &6  &2  &8  &  &  &   \\
    &7  &16 &0  &16 &  &  &   \\
    &8  &8  &0  &8  &  &  &   \\
12  &9  &4  &0  &4  &38&2 &40 \\
    &10 &2  &0  &2  &  &  &   \\
    &11 &1  &0  &1  &  &  &   \\
    &13 &1  &0  &1  &  &  &   \\ \hline
    &2  &1  &0  &1  &  &  &   \\ 
    &3  &2  &1  &3  &  &  &   \\
    &4  &6  &2  &8  &  &  &   \\
    &5  &9  &5  &14 &  &  &   \\
    &6  &13 &3  &16 &  &  &   \\
13  &7  &22 &10 &32 &85&21&106\\
    &8  &16 &0  &16 &  &  &   \\
    &9  &8  &0  &8  &  &  &   \\
    &10 &4  &0  &4  &  &  &   \\
    &11 &2  &0  &2  &  &  &   \\
    &12 &1  &0  &1  &  &  &   \\
    &14 &1  &0  &1  &  &  &   \\  \hline
    &3  &3  &0  &3  &  &  &   \\
    &4  &5  &1  &6  &  &  &   \\
    &5  &12 &2  &14 &  &  &   \\
    &6  &14 &2  &16 &  &  &   \\
    &8  &32 &0  &32 &  &  &   \\
14  &9  &16 &0  &16 &98&5 &103\\
    &10 &8  &0  &8  &  &  &   \\
    &11 &4  &0  &4  &  &  &   \\
    &12 &2  &0  &2  &  &  &   \\
    &13 &1  &0  &1  &  &  &   \\
    &15 &1  &0  &1  &  &  &   \\ \hline
    &2  &1  &0  &1  &  &  &   \\
    &4  &9  &2  &11 &  &  &   \\
    &6  &21 &7  &28 &  &  &   \\
    &7  &25 &7  &32 &  &  &   \\
    &8  &47 &17 &64 &  &  &   \\
15  &9  &32 &0  &32 &167&33&200\\
    &10 &16 &0  &16 &  &  &   \\
    &11 &8  &0  &8  &  &  &   \\
    &12 &4  &0  &4  &  &  &   \\
    &13 &2  &0  &2  &  &  &   \\
    &14 &1  &0  &1  &  &  &   \\
    &16 &1  &0  &1  &  &  &   \\ \hline
    &3  &3  &0  &3  &  &  &   \\
    &5  &11 &3  &14 &  &  &   \\
    &6  &23 &5  &28 &  &  &   \\
    &7  &28 &4  &32 &  &  &   \\
    &9  &64 &0  &64 &  &  &   \\
16  &10 &32 &0  &32 &193&12&205\\
    &11 &16 &0  &16 &  &  &   \\
    &12 &8  &0  &8  &  &  &   \\
    &13 &4  &0  &4  &  &  &   \\
    &14 &2  &0  &2  &  &  &   \\
    &15 &1  &0  &1  &  &  &   \\ 
    &17 &1  &0  &1  &  &  &   \\ \hline
    &2  &1  &0  &1  &  &  &   \\
    &3  &3  &1  &4  &  &  &   \\
    &4  &9  &3  &12 &  &  &   \\
    &5  &15 &7  &22 &  &  &   \\
    &6  &36 &14 &50 &  &  &   \\
    &7  &43 &13 &56 &  &  &   \\
    &8  &52 &12 &64 &  &  &   \\
17  &9  &92 &36 &128&379&86&465\\
    &10 &64 &0  &64 &  &  &   \\
    &11 &32 &0  &32 &  &  &   \\
    &12 &16 &0  &16 &  &  &   \\
    &13 &8  &0  &8  &  &  &   \\
    &14 &4  &0  &4  &  &  &   \\
    &15 &2  &0  &2  &  &  &   \\
    &16 &1  &0  &1  &  &  &   \\
    &18 &1  &0  &1  &  &  &   \\ \hline
    &4  &7  &2  &9  &  &  &   \\
    &5  &18 &2  &20 &  &  &   \\
    &7  &46 &10 &56 &  &  &   \\
    &8  &56 &8  &64 &  &  &   \\
    &10 &128&0  &128&  &  &   \\
    &11 &64 &0  &64 &  &  &   \\
18  &12 &32 &0  &32 &383&22&405\\
    &13 &16 &0  &16 &  &  &   \\
    &14 &8  &0  &8  &  &  &   \\
    &15 &4  &0  &4  &  &  &   \\
    &16 &2  &0  &2  &  &  &   \\
    &17 &1  &0  &1  &  &  &   \\
    &19 &1  &0  &1  &  &  &   \\ \hline
    &2  &1  &0  &1  &  &  &   \\
    &3  &3  &1  &4  &  &  &   \\
    &4  &13 &3  &16 &  &  &   \\
    &5  &21 &10 &31 &  &  &   \\
    &6  &41 &16 &57 &  &  &   \\
    &7  &67 &33 &100&  &  &   \\
    &8  &93 &19 &112&  &  &   \\
    &9  &102&26 &128&  &  &   \\
19  &10 &190&66 &256&787&174&961\\
    &11 &128&0  &128&  &  &   \\
    &12 &64 &0  &64 &  &  &   \\
    &13 &32 &0  &32 &  &  &   \\
    &14 &16 &0  &16 &  &  &   \\
    &15 &8  &0  &8  &  &  &   \\
    &16 &4  &0  &4  &  &  &   \\
    &17 &2  &0  &2  &  &  &   \\
    &18 &1  &0  &1  &  &  &   \\
    &20 &1  &0  &1  &  &  &   \\ \hline
    &3  &4  &0  &4  &  &  &   \\
    &6  &39 &9  &48 &  &  &   \\
    &7  &73 &23 &96 &  &  &   \\
    &8  &96 &16 &112&  &  &   \\
    &9  &114&14 &128&  &  &   \\
    &11 &256&0  &256&  &  &   \\
    &12 &128&0  &128&  &  &   \\
20  &13 &64 &0  &64 &838&62&900\\
    &14 &32 &0  &32 &  &  &   \\
    &15 &16 &0  &16 &  &  &   \\
    &16 &8  &0  &8  &  &  &   \\
    &17 &4  &0  &4  &  &  &   \\
    &18 &2  &0  &2  &  &  &   \\
    &19 &1  &0  &1  &  &  &   \\
    &21 &1  &0  &1  &  &  &   \\ \hline
    &2  &1  &0  &1  &  &  &   \\
    &4  &14 &3  &17 &  &  &   \\
    &5  &22 &9  &31 &  &  &   \\
    &6  &58 &17 &75 &  &  &   \\
    &8  &153&47 &200&  &  &   \\
    &9  &183&41 &224&  &  &   \\
    &10 &210&46 &256&  &  &   \\
    &11 &376&136&512&  &  &   \\
21  &12 &256&0  &256&1529&299&1828\\
    &13 &128&0  &128&  &  &   \\
    &14 &64 &0  &64 &  &  &   \\
    &15 &32 &0  &32 &  &  &   \\
    &16 &16 &0  &16 &  &  &   \\
    &17 &8  &0  &8  &  &  &   \\
    &18 &4  &0  &4  &  &  &   \\
    &19 &2  &0  &2  &  &  &   \\
    &20 &1  &0  &1  &  &  &   \\
    &22 &1  &0  &1  &  &  &   \\ \hline
    &3  &4  &0  &4  &  &  &   \\
    &4  &10 &2  &12 &  &  &   \\
    &5  &26 &6  &32 &  &  &   \\
    &6  &56 &12 &68 &  &  &   \\
    &7  &84 &17 &101&  &  &   \\
    &8  &156&36 &192&  &  &   \\
    &9  &194&30 &224&  &  &   \\
    &10 &227&29 &256&  &  &   \\
    &12 &512&0  &512&  &  &   \\
22  &13 &256&0  &256&1781&132&1913\\
    &14 &128&0  &128&  &  &   \\
    &15 &64 &0  &64 &  &  &   \\
    &16 &32 &0  &32 &  &  &   \\
    &17 &16 &0  &16 &  &  &   \\
    &18 &8  &0  &8  &  &  &   \\
    &19 &4  &0  &4  &  &  &   \\
    &20 &2  &0  &2  &  &  &   \\
    &21 &1  &0  &1  &  &  &   \\
    &23 &1  &0  &1  &  &  &   \\ \hline
    &2  &1  &0  &1   &  &  &   \\
    &3  &4  &1  &5   &  &  &   \\
    &4  &18 &4  &22  &  &  &   \\
    &5  &30 &13 &43  &  &  &   \\
    &6  &91 &31 &122 &  &  &   \\
    &7  &111&41 &152 &  &  &   \\
    &8  &253&90 &343 &  &  &   \\
    &9  &301&99 &400 &  &  &   \\
    &10 &375&73 &448 &  &  &   \\
    &11 &416&96 &512 &  &  &   \\
    &12 &766&258&1024&  &  &   \\
23  &13 &512&0  &512 &3390&706&4096\\
    &14 &256&0  &256 &  &  &   \\
    &15 &128&0  &128 &  &  &   \\
    &16 &64 &0  &64  &  &  &   \\
    &17 &32 &0  &32  &  &  &   \\
    &18 &16 &0  &16  &  &  &   \\
    &19 &8  &0  &8   &  &  &   \\
    &20 &4  &0  &4   &  &  &   \\
    &21 &2  &0  &2   &  &  &   \\
    &22 &1  &0  &1   &  &  &   \\
    &24 &1  &0  &1   &  &  &   \\ \hline
    &5  &38  &4 &42  &  &  &   \\
    &7  &115 &29&144 &  &  &   \\
    &9  &320 &64&384 &  &  &   \\
    &10 &402 &46&448 &  &  &   \\
    &11 &455 &57&512 &  &  &   \\
    &13 &1024&0 &1024&  &  &   \\
    &14 &512 &0 &512 &  &  &   \\
    &15 &256 &0 &256 &  &  &   \\
 24 &16 &128 &0 &128 &3378&200&3578   \\
    &17 &64  &0 &64  &  &  &   \\
    &18 &32  &0 &32  &  &  &   \\
    &19 &16  &0 &16  &  &  &   \\
    &20 &8   &0 &8   &  &  &   \\
    &21 &4   &0 &4   &  &  &   \\
    &22 &2   &0 &2   &  &  &   \\
    &23 &1   &0 &1   &  &  &   \\
    &25 &1   &0 &1   &  &  &   \\ \hline
    &2  &1   &0  &1   &  &  &   \\
    &3  &4   &1  &5   &  &  &   \\
    &4  &19  &4  &23  &  &  &   \\
    &6  &98  &34 &132 &  &  &   \\
    &7  &162 &66 &228 &  &  &   \\
    &8  &288 &94 &382 &  &  &   \\
    &9  &483 &203&686 &  &  &   \\
    &10 &641 &159&800 &  &  &   \\
    &11 &744 &152&896 &  &  &   \\
    &12 &845 &179&1024&  &  &   \\
    &13 &1522&526&2048&  &  &   \\
25  &14 &1024&0  &1024&6855&1418&8273 \\
    &15 &512 &0  &512 &  &  &   \\
    &16 &256 &0  &256 &  &  &   \\
    &17 &128 &0  &128 &  &  &   \\
    &18 &64  &0  &64  &  &  &   \\
    &19 &32  &0  &32  &  &  &   \\
    &20 &16  &0  &16  &  &  &   \\
    &21 &8   &0  &8   &  &  &   \\
    &22 &4   &0  &4   &  &  &   \\
    &23 &2   &0  &2   &  &  &   \\
    &24 &1   &0  &1   &  &  &   \\
    &26 &1   &0  &1   &  &  &   \\ \hline
    &3  &5   &0  &5   &  &  &   \\
    &4  &13  &3  &16   &  &  &   \\
    &5  &38  &6  &44  &  &  &   \\
    &6  &85  &24 &109 &  &  &   \\
    &7  &179 &37 &216 &  &  &   \\
    &8  &273 &61 &334 &  &  &   \\
    &9  &529 &138&667 &  &  &   \\
    &10 &651 &117&768 &  &  &   \\
    &11 &807 &89 &896 &  &  &   \\
    &12 &912 &112&1024&  &  &   \\
    &14 &2048&0  &2048&  &  &   \\
26  &15 &1024&0  &1024&  &  &   \\
    &16 &512 &0  &512 &  &  &   \\
    &17 &256 &0  &256 &  &  &   \\
    &18 &128 &0  &128 &  &  &   \\
    &19 &64  &0  &64  &  &  &   \\
    &20 &32  &0  &32  &  &  &   \\
    &21 &16  &0  &16  &  &  &   \\
    &22 &8   &0  &8   &  &  &   \\
    &23 &4   &0  &4   &  &  &   \\
    &24 &2   &0  &2   &  &  &   \\
    &25 &1   &0  &1   &  &  &   \\
    &27 &1   &0  &1   &  &  &   \\ \hline
    &   &   &   &   &  &  &   \\
    &   &   &   &   &  &  &   \\
    &   &   &   &   &  &  &   \\
    &   &   &   &   &  &  &   \\
    &   &   &   &   &  &  &   \\
    &   &   &   &   &  &  &   \\
    &   &   &   &   &  &  &   \\
    
\hline
\end{longtable}
\end{center}

\clearpage

\end{document}